\documentclass[12pt]{amsart}

\usepackage[english]{babel}
\usepackage{vmargin,graphicx,amsmath,amssymb,color,subfigure,enumerate,csquotes, dsfont, mathrsfs}
\numberwithin{equation}{section}

\usepackage{color}
\definecolor{gr}{rgb}   {0.,   0.69,   0.23 }
\definecolor{bl}{rgb}   {0.,   0.5,   1. }
\definecolor{mg}{rgb}   {0.85,  0.,    0.85}
\definecolor{or}{rgb}   {0.9,  0.5,   0.}

\definecolor{webred}{rgb}{0.75,0,0}
\definecolor{webgreen}{rgb}{0,0.75,0}
\usepackage[citecolor=webgreen,colorlinks=true,linkcolor=webred]{hyperref}

\newtheorem{theorem}{Theorem}[section]
\newtheorem{lemma}[theorem]{Lemma}

\newtheorem{remark}[theorem]{Remark}

\newtheorem{proposition}[theorem]{Proposition}
\newtheorem{corollary}[theorem]{Corollary}

\newtheorem{definition}[theorem]{Definition\rm}

\newcommand{\dist}{\mathsf{dist}}
\renewcommand{\Re}{\mathsf{Re}}
\newcommand{\loc}{\mathsf{loc}}
\newcommand{\e}{\mathsf{e}}
\newcommand{\x}{\mathbf{x}}
\newcommand{\y}{\mathbf{y}}
\newcommand{\n}{\mathbf{n}}
\newcommand{\Dom}{\mathsf{Dom}}
\newcommand{\supp}{\mathsf{supp}\,}
\newcommand{\spe}{\mathsf{sp}_{\mathsf{ess}}}

\def\Q{\mathfrak{H}}
\def\Qk{\Q^{(k)}}
\def\eig#1#2{\lambda_{#1,#2}}
\def\eigone#1{\eig{1}{#1}}

\newcommand{\N}{\mathbb{N}}
\newcommand{\Z}{\mathbb{Z}}
\newcommand{\R}{\mathbb{R}}

\newcommand{\A}{\mathbf{A}}
\newcommand{\B}{\mathbf{B}}

\newcommand{\Neu}{\mathbf{Neu}}
\newcommand{\sH}{\mathsf{H}}
\newcommand{\sL}{\mathsf{L}}
\newcommand{\eps}{\varepsilon}
\newcommand{\dx}{\,\mathrm{d}}

\begin{document}

\title[Optimal magnetic Sobolev constants]{Optimal magnetic Sobolev constants\\ in the semiclassical limit}

\author{S. Fournais}
\address[S. Fournais]{Aarhus University, Ny Munkegade 118, DK-8000 Aarhus~C, Denmark}
\email{fournais@math.au.dk}
\date{\today}
\author{N. Raymond}
\address[N. Raymond]{IRMAR, Universit\'e de Rennes 1, Campus de Beaulieu, F-35042 Rennes cedex, France}
\email{nicolas.raymond@univ-rennes1.fr}

\maketitle
\begin{abstract}
This paper is devoted to the semiclassical analysis of the best constants in the magnetic Sobolev embeddings in the case of a bounded domain of the plane carrying Dirichlet conditions. We provide quantitative estimates of these constants (with an explicit dependence on the semiclassical parameter) and analyze the exponential localization in $\sL^\infty$-norm of the corresponding minimizers near the magnetic wells.
\end{abstract}

\section{Preliminary considerations and main results}

We are interested in the following minimization problem. For the sake of simplicity, we consider a simply connected bounded domain $\Omega\subset \R^2$, $p\in[2,+\infty)$, $h>0$
and a smooth vector potential $\A$ on $\overline{\Omega}$. We introduce the following \enquote{nonlinear eigenvalue} (or optimal magnetic Sobolev constant):
\begin{equation}\label{inf}
\lambda(\Omega, \A, p,h)=\inf_{\psi\in \sH^1_{0}(\Omega), \psi\neq 0}\frac{\mathcal{Q}_{h,\A}(\psi)}{\left(\int_{\Omega}|\psi|^p\dx x\right)^{\frac{2}{p}}}=\inf_{\underset{ \|\psi\|_{\sL^p(\Omega)}=1}{\psi\in \sH^1_{0}(\Omega),}}\mathcal{Q}_{h,\A}(\psi),
\end{equation}
where the magnetic quadratic form is defined by
$$\forall \psi\in\sH^1_{0}(\Omega),\quad\mathcal{Q}_{h,\A}(\psi)=\int_{\Omega}|(-ih\nabla+\A)\psi|^2\dx\x.$$
The Dirichlet realization of the magnetic Laplacian on $\Omega$ (defined as the Friedrichs extension) is denoted by $\mathcal{L}_{h,\A}$ whose domain is
$$\Dom\left(\mathcal{L}_{h,\A}\right)=\{\psi\in\Dom(\mathcal{Q}_{h,\A})=\sH^1_{0}(\Omega) : (-ih\nabla+\A)^2\psi\in\sL^2(\Omega)\}.$$ 
We recall that $\B=\nabla\times\A$ is called the magnetic field (with the notation $\A = (A_1, A_2)$, $\B = \partial_2 A_1 - \partial_1 A_2$). Let us here notice that there exists a vast literature dealing with the case $p=2$. In this case $\lambda(\Omega, \A, 2,h)$ is the lowest eigenvalue of the magnetic Laplacian. On this subject, the reader may consult the books and reviews \cite{FouHel10, HelKo14, Ray14}.
\subsection{Motivations and context}
Before describing the motivations of this paper, let us recall some basic facts concerning the minimization problem \eqref{inf}.
\begin{lemma}
The infimum in \eqref{inf} is attained.
\end{lemma}
\begin{proof}
The proof is standard but we recall it for completeness. Consider a minimizing sequence $(\psi_{j})$ that is normalized in $\sL^p$-norm. 
Then, by a H\"{o}lder inequality and using that $\Omega$ has bounded measure, $(\psi_{j})$ is bounded in $\sL^2$.
Since $\A \in \sL^{\infty}(\Omega)$, we conclude that $(\psi_{j})$ is bounded in $\sH^1_{0}(\Omega)$.
By the Banach-Alaoglu Theorem there exists a subsequence (still denoted by $(\psi_{j})$) and $\psi_{\infty} \in \sH^1_{0}(\Omega)$ such that $\psi_j \rightharpoonup \psi_{\infty}$ weakly in $\sH^1_{0}(\Omega)$ and $\psi_j \rightarrow \psi_{\infty}$ in $\sL^q(\Omega)$ for all $q\in[2,+\infty)$.
This is enough to conclude.
\end{proof}

Let us now consider the (focusing) equation satisfied by the minimizers.

\begin{lemma}
The minimizers (which belong to $\sH^1_{0}(\Omega)$) of the $\sL^p$-normalized version of \eqref{inf} satisfy the following equation in the sense of distributions:
\begin{equation}\label{NLS}
(-ih\nabla+\A)^2\psi=\lambda(\Omega, \A, p, h)|\psi|^{p-2}\psi,\qquad \|\psi\|_{\sL^p(\Omega)}=1.
\end{equation}
In particular (by Sobolev embedding), the minimizers belong to the domain of $\mathcal{L}_{h,\A}$.
\end{lemma}
This paper is motivated by the seminal paper \cite{EL89} where the minimization problem \eqref{inf} is investigated for $\Omega=\R^d$ and with a constant magnetic field (and also in the case of some nicely varying magnetic fields). In particular, Esteban and Lions prove the existence of minimizers by using the famous concentration-compactness method. In the present paper, we want to describe the minimizers as well as the infimum itself in the semiclassical limit $h\to 0$. The naive idea is that, locally, modulo a blow up argument, they should look like the minimizers in the whole plane. In our paper, we will also allow the magnetic field to vanish and this will lead to other minimization problems in the whole plane which are interesting in themselves and for which the results of \cite{EL89} do not apply.

Another motivation to consider the minimization problem \eqref{inf} comes from the recent paper \cite{diCS14}. Di Cosmo and Van Schaftingen analyze a close version of \eqref{inf} in $\Omega\subset\R^d$ in the presence of an additional electric potential. Note here that, as for the semiclassical analysis in the case $p=2$, if there is a non-zero electric potential, then the minima of $V$ tend to attract the bound states, independently from the presence of the magnetic field \cite{K00}; in \cite{diCS14} the electric potential is multiplied by $h$ and plays on the same scale as $\B$ which cannot be treated as a perturbation. These authors prove, modulo subsequences extractions of the semiclassical parameter, that the asymptotics of the optimal Sobolev constant (with electro-magnetic field) is governed by a family of model minimization problems with constant magnetic fields and electric potentials. They also establish that we can find minimizers of \eqref{inf} which are localized near the minima of the \enquote{concentration function}. This function is nothing but the infimum of the model problem in $\R^d$, depending on the point $\x$ where the blow up occurs. In all their estimates, these authors do not quantify the convergence with an explicit dependence on $h$. 

In our paper we will especially tackle this question in dimension two in the case of a pure magnetic field and we will see how this refinement can be applied to get localization estimates. As in \cite{diCS14}, we will first give upper bounds of $\lambda(\Omega,\A,p,h)$, but with quantitative remainders. This will be the constructive and explicit part of the analysis, relying on the model minimization problems. Note that this is the only part of our analysis where the concentration-compactness method will be used, to analyze model problems with vanishing magnetic fields. Then we will establish lower bounds and localization estimates. For that purpose, inspired by the linear spectral techniques (see for instance \cite[Part I]{FouHel10}), we will provide an alternative point of view to the semiclassical concentration-compactness arguments of \cite{diCS14} by using semiclassical partitions of unity adapted simultaneously to the $\sL^p$-norm and the magnetic quadratic form. Moreover, in the case of non-vanishing magnetic fields, we will establish exponential decay estimates of all the minimizers of \eqref{inf} away from the magnetic wells (the minima of $\B$). In fact, we will use the philosophy of the semiclassical linear methods: the more accurate the estimate of $\lambda(\Omega,\A,p,h)$ obtained, the more refined is the localization of the bound states. A very rough localization estimate of the bound states in $\sL^p$-norm (directly related to the remainders in the estimates of $\lambda(\Omega,\A,p,h)$) will then be enough to get an a priori control of the nonlinearity and the investigation will be reduced to the well-known semiclassical concentration estimates \`a la Agmon (see \cite{Agmon85, HelSj84, Hel88}), jointly with standard elliptic estimates. Finally, note that our investigation deals with the magnetic analog of the pure electric case of \cite{W93} (see also \cite{dPF97}). We could include in our analysis an electric potential, but we refrain to do so to highlight the pure magnetic effects.

\subsection{Results}
We would like to provide an accurate description of the behavior of $\lambda(\Omega,\A,p,h)$ when $h$ goes to zero.
Locally, we can approximate by either a constant magnetic field, or a magnetic field having a zero of a certain order. Therefore, we introduce the following notation. 

\begin{definition}
For $k\in\N$, we define
\begin{equation}\label{lambda0}
\lambda^{[k]}(p )=\lambda(\R^2, \A^{[k]}, p,1)=\inf_{\psi\in \Dom(Q_{\A^{[k]}}), \psi\neq 0}\frac{Q_{\A^{[k]}}(\psi)}{\|\psi\|^2_{\sL^p}},
\end{equation}
where $\A^{[k]}(x,y)=\left(0, \frac{x^{k+1}}{k+1}\right)$. Here 
\begin{align*}
Q_{\A^{[k]}}(\psi) = \int_{\R^2} |(-i\nabla +\A^{[k]})\psi|^2\dx \x,
\end{align*}
with domain
\begin{align*}
\Dom(Q_{\A^{[k]}})=\left\{ \psi \in \sL^2(\R^2)\,:\, (-i\nabla +\A^{[k]})\psi\in \sL^2(\R^2) \right\}.
\end{align*}
\end{definition}
In the case $k=0$ and $p\geq2$, it is known that the infimum is a minimum (see \cite{EL89}). We will prove in this paper that, for $k\geq 1$ and $p>2$, the minimum is also attained, even if the corresponding magnetic field does not satisfy the assumptions of \cite{EL89}.

We can now state our first theorem concerning the case when the magnetic field does not vanish.
\begin{theorem}\label{theo1}
Let $p\geq 2$. Let us assume that $\A$ is smooth on $\overline{\Omega}$, that $\B=\nabla \times \A$ does not vanish on $\overline{\Omega}$ and that its minimum $b_{0}$ is attained in $\Omega$. Then there exist $C>0$ and $h_{0}>0$ such that, for all $h\in(0,h_{0})$,
$$(1-Ch^{\frac{1}{8}})\lambda^{[0]}(p )b^{\frac{2}{p}}_{0}h^2h^{-\frac{2}{p}}\leq\lambda(\Omega,\A, p, h)\leq (1+Ch^{1/2})\lambda^{[0]}(p)b_{0}^{\frac{2}{p}}h^2h^{-\frac{2}{p}}.$$
Moreover, if the magnetic field is only assumed to be smooth and positive on $\overline{\Omega}$ (with a minimum possibly on the boundary), the lower bound is still valid.
\end{theorem}

\begin{remark}
The error estimate in the upper bound in Theorem~\ref{theo1} matches the corresponding bound in the well-known linear case and we expect it to be optimal. However, the relative error of $h^{\frac{1}{8}}$ in the lower bound is unlikely to be best possible. The same remark applies to the error bounds in Theorem~\ref{theo2}.
\end{remark}

In the following theorem, we state an exponential concentration property of the minimizers.
\begin{theorem}\label{loc}
Let $p>2$, $\rho\in\left(0,\frac{1}{2}\right)$, $\eps>0$ and consider the same assumptions as in Theorem \ref{theo1} and also assume that the minimum is unique and attained at $\x_{0} \in \Omega$. 

Then there exist $C>0$ and $h_{0}>0$ such that, for all $h\in(0,h_{0})$ and all $\psi$ solution of \eqref{NLS}, we have
$$\|\psi\|_{\sL^\infty(\complement D(\x_{0},2\eps))}\leq Ce^{-Ch^{-\rho}}\|\psi\|_{\sL^\infty(\Omega)},$$
where $D(\x,R)$ denotes the open ball of center $\x$ and radius $R>0$.
\end{theorem}
\begin{remark}
In Theorem \ref{loc}, if the minimum of the magnetic field is non-degenerate, we can replace $\eps$ by $h^{\gamma}$ with $\gamma>0$ sufficiently small. In Theorem \ref{loc}, we have the same kind of results in the case of multiple wells. Theorems \ref{theo1} and \ref{loc} are quantitative improvements of \cite[Theorem 1.1]{diCS14} in the pure magnetic case. We can notice that, when $p>2$, we have
\begin{equation}\label{NLS-r}
(-ih\nabla+\A)^2\varphi=|\varphi|^{p-2}\varphi,
\end{equation}
with $\varphi=\lambda(\Omega,\A,p,h)^{\frac{1}{p-2}}\psi$. Thus, we have constructed solutions of \eqref{NLS-r} which decay exponentially away from the magnetic wells in the semiclassical limit. 
\end{remark}
The following theorem analyzes the case when the magnetic field vanishes along a smooth curve.
\begin{theorem}\label{theo2}
Let $p>2$. Let us assume that $\A$ is smooth on $\overline{\Omega}$,
that 
$$
\Gamma = \{ \x \in \overline{\Omega}\,:\,  \B(\x) = 0 \},
$$
satisfies that $\Gamma \subset \Omega$ is a smooth, simple and closed curve, and that $\B$ vanishes non-degenerately along $\Gamma$ in the sense that
$$
\nabla \B(\x) \neq 0, \qquad \text{ for all } \x \in \Gamma.
$$
Assuming that $\B$ is positive inside $\Gamma$ and negative outside, we denote by $\gamma_{0}>0$ the minimum of the normal derivative of $\B$ with respect to $\Gamma$. Then there exist $C>0$ and $h_{0}>0$ such that, for all $h\in(0,h_{0})$,
$$(1-Ch^{\frac{1}{33}})\lambda^{[1]}(p )\gamma_{0}^{\frac{4}{3p}}h^2 h^{-\frac{4}{3p}}\leq\lambda(\Omega,\A, p, h)\leq(1+Ch^{\frac{1}{3}})\lambda^{[1]}(p )\gamma_{0}^{\frac{4}{3p}}h^2 h^{-\frac{4}{3p}}.$$
\end{theorem}
\begin{remark}
The case $p=2$ is treated in \cite{DomRay13} (see also \cite{Montgomery95, HelMo96}). In \cite{diCS14}, it is only stated that $h^{-2+\frac{2}{p}}\lambda(\Omega,\A,p,h)$ goes to zero when $h$ goes to zero. Moreover, by using the strategy of the proof of Theorem \ref{loc}, one can establish an exponential decay of the ground states away from $\Gamma$.
\end{remark}
\subsection{Organization of the paper}
In Section \ref{Models}, we investigate the existence of minimizers of \eqref{lambda0} and their decay properties. In Section \ref{Upper}, we provide the upper bounds stated in Theorems \ref{theo1} and \ref{theo2}. Section \ref{Lower} is devoted to the analysis of the corresponding lower bounds and to the proof of Theorem \ref{loc}.

\section{Minimizers of the models and exponential decay}\label{Models}

\subsection{Existence of bound states and exponential decay}
We first recall the diamagnetic inequality and the so-called \enquote{IMS} formula (see \cite{CFKS87, FouHel10}).
\begin{lemma}\label{lem:diamagn}
We have, for $u\in\Dom(Q_{\A})$,
$$|\nabla |u||\leq |(-i\nabla+\A)u|,\quad \mbox{a.e.}$$
which implies that
$$\|\nabla |u|\|^2_{\sL^2(\R^2)}\leq Q_{\A}(u),$$
\end{lemma}
\begin{lemma}\label{IMS}
If $\chi$ is a Lipschitzian function and $u\in\Dom(\mathcal{L}_{h,\A})$, then we have
\begin{align}\label{eq:IMS}
\Re\langle\mathcal{L}_{h,\A}u,\chi^2 u\rangle_{\sL^2(\Omega)}=\mathcal{Q}_{h,\A}(\chi u)-h^2\|\nabla\chi\, u\|^2_{\sL^2(\Omega)}
\end{align}
\end{lemma}
Finally, we recall the following useful lower bound (see for instance \cite{AHS}),
\begin{align}\label{eq:MagnUncertain}
\mathcal{Q}_{h,\A}(u) \geq h \int_{\Omega} \B |u|^2 \dx \x
\end{align}
for all $u \in \sH^1_{0}(\Omega)$.

We recall that we can define the Friedrichs extension of the electro-magnetic Laplacian as soon as the electric potential belongs to some $\sL^q$ space.
\begin{proposition}\label{extension}
Let us consider $\A\in\mathcal{C}^\infty(\R^2)$ and $\displaystyle{V\in \sL^q(\R^2)}$, for some $q>1$.
For all $u\in\Dom(Q_{\A})$, we may define
$$Q_{\A,V}(u)=\int_{\R^2} |(-i\nabla+\A)u|^2\dx \x+\int_{\R^2} V|u|^2\dx \x.$$
Then, for all $\eps>0$, there exists  $C>0$ such that
\begin{equation}\label{eq1}
\forall u\in\Dom(Q_{\A}),\qquad Q_{\A,V}(u)\geq (1-\eps)Q_{\A}(u)-C\|u\|^2_{\sL^2(\R^2)},
\end{equation}
Furthermore, for all $\eps>0$ there exists $R>0$ such that $\forall u\in\Dom(Q_{\A})$ with $\supp u\subset \complement D(0,R)$,
\begin{equation}\label{eq2}
 Q_{\A,V}(u)\geq (1-\eps)Q_{\A}(u)-\eps\|u\|^2_{\sL^2(\R^2)}.
\end{equation}
Moreover, we may define the self-adjoint operator---the Friedrichs extension---$\mathcal{L}_{\A,V}$ of $Q_{\A,V}$ whose domain is
$$\Dom(\mathcal{L}_{\A,V})=\left\{u\in\Dom(Q_{\A}) : \left((-i\nabla+\A)^2+V\right)u\in\sL^2(\R^2) \right\},$$ and
$$
\mathcal{L}_{\A,V} u = \left((-i\nabla+\A)^2+V\right)u,
$$
for all $u \in \Dom(\mathcal{L}_{\A,V})$.
\end{proposition}

\begin{proof}
Let us recall the Sobolev embedding $\sH^1(\R^2)\subset \sL^{r}(\R^2)$: For all $r\geq 2$ there exist $C(r )>0$ such that for all $u\in\sH^1(\R^2)$,
\begin{align}
\label{eq:Sobolev}
\|u\|_{\sL^r(\R^2)}\leq C(r )(\|u\|_{\sL^2(\R^2)}+\|\nabla u\|_{\sL^2(\R^2)}).
\end{align}
In particular, for all $v\in\sH^1(\R^2)$ and all $\eps>0$, we apply this inequality to the rescaled function $u_{\eps}(x)=v(\eps^{\frac{r}{2}} x)$ to infer the rescaled version of the Sobolev embedding
\begin{equation}\label{re-Sob}
\|v\|_{\sL^r(\R^2)}\leq C(r )(\eps^{1-\frac{r}{2}}\|v\|_{\sL^2(\R^2)}+\eps\|\nabla v\|_{\sL^2(\R^2)}).
\end{equation}
With the diamagnetic inequality, this implies that, for all $u\in\Dom(Q_{\A})$, we have $u\in \sL^r(\R^2)$ for all $r\geq 2$ and $Q_{\A,V}(u)$ is well defined. Then let us prove \eqref{eq1}. We use the Cauchy-Schwarz inequality to get, for all $u\in\Dom(Q_{\A})$,
$$\left|\int_{\R^2} V |u|^2\dx x\right|\leq \|V\|_{\sL^q(\R^2}\|u^2\|_{\sL^{q'}(\R^2)}= \|V\|_{\sL^q(\R^2}\|u\|^2_{\sL^{2q'}(\R^2)},$$
where $\frac{1}{q}+\frac{1}{q'}=1$. Since $q>1$, we have $1<q'<+\infty$ so that with \eqref{re-Sob},
$$\|u\|^2_{\sL^{2q'}(\R^2)}\leq \tilde C(q' )(\eps^{1-q'}\|u\|^2_{\sL^2(\R^2)}+\eps\|\nabla |u|\|^2_{\sL^2(\R^2)})$$
and so
$$\|u\|^2_{\sL^{2q'}(\R^2)}\leq \tilde C(q' )(\eps^{1-q'}\|u\|^2_{\sL^2(\R^2)}+\eps Q_{A}(u)+\eps\|u\|^2_{\sL^2(\R^2)})$$
and \eqref{eq1} follows as well as the existence of the Friedrichs extension $\mathcal{L}_{\A,V}$. 

Let us now prove \eqref{eq2}. For all $u\in\Dom(Q_{\A})$ supported in $\complement D(0,R)$,
$$\left|\int_{\R^2} V |u|^2\dx x\right|\leq \|V\|_{\sL^q(\complement D(0,R))}\|u\|^2_{\sL^{2q'}(\R^2)},$$
and we deduce
\begin{multline*}
\left|\int_{\R^2} V |u|^2\dx x\right|\leq C^2\|V\|_{\sL^q(\complement D(0,R)}\| |u| \|_{\sH^1(\R^2)}^2\\
=C^2\|V\|_{\sL^q(\complement D(0,R))}\left(\|u\|^2_{\sL^2(\R^2)}+ \int_{\R^2} |(-i\nabla+\A)u|^2\dx x
\right).
\end{multline*}
It remains to use that $V\in\sL^q(\R^2)$ and to take $R$ large enough to get \eqref{eq2}. 
\end{proof}

\begin{proposition}\label{Agmon}
For $k=0$ and $p\geq 2$, the infimum in \eqref{lambda0} is a minimum. Moreover, if $\psi$ is a minimizer, there exist $\eta, C>0$ such that 
\begin{align}
\int e^{2\eta |\x|} ( |\psi(\x)|^2 + |(-i\nabla+\A(\x))\psi)(\x)|^2)\dx x \leq C \| \psi \|_{\sL^2(\R^2)}^2. 
\end{align}
\end{proposition}

\begin{proof}
The fact that the infimum is attained is proved in \cite{EL89}. Let $\psi$ be a minimizer such that $\|\psi\|_{\sL^p(\R^2)}=1$.

We introduce the potential $V=-\lambda_{0}|\psi|^{p-2}\leq 0$ which---by \eqref{eq:Sobolev}---belongs to $\sL^q(\R^2)$ for all $q\geq \frac{2}{p-2}$. 
By Proposition \ref{extension}, we may consider the electro-magnetic Laplacian $\mathcal{L}_{\A^{[0]},V}$ defined as the Friedrichs extension of the quadratic form
$$Q_{\A^{[0]},V}(u)=\int_{\R^2}|(-i\nabla+\A^{[0]})u|^2\dx \x+\int_{\R^2}V|u|^2\dx \x,\qquad\forall u\in\mathcal{C}_{0}^\infty(\R^2).$$
We notice that $\psi\in\Dom(\mathcal{L}_{\A^{[0]},V})$, $\psi\neq 0$ and $\mathcal{L}_{\A^0,V}\psi=0$. 
With \eqref{eq2} in Proposition \ref{extension}, for all $\eps>0$, there exists $R_{0}>0$ such that, for all $u\in\Dom(Q_{\A^{[0]}})$, such that $\supp(u)\subset \complement D(0,R_{0})$, we have
$$Q_{\A^{[0]},V}(u)\geq (1-\eps)Q_{\A^{[0]}}(u)-\eps\|u\|^2_{\sL^2(\R^2)}.$$
But we have $Q_{\A^{[0]}}(u)\geq \|u\|^2_{\sL^2(\R^2)}$, so that
$$Q_{\A^{[0]},V}(u)\geq (1-2\eps)\|u\|^2_{\sL^2(\R^2)}.$$
From Persson's theorem, we infer that $\inf\spe(\mathcal{L}_{\A^{[0]},V})\geq 1$---with $\spe$ denoting the essential spectrum. Now, by definition, $\psi$ is an eigenfunction associated with the eigenvalue $0<1$ and, by Agmon estimates (see \cite{Persson60, Agmon85}), it has an exponential decay.
\end{proof}

\begin{corollary}\label{cor:AgmonL-p}
Let $k \in \N$ and 
let $\psi$ be a minimizer of \eqref{lambda0}.
\begin{itemize}
\item For any $q \in [2, \infty)$, we have $\| \psi \|_{\sL^2(\R^2)} \leq C_q \| \psi \|_{\sL^q(\R^2)} \leq \widetilde C_q \| \psi \|_{\sL^2(\R^2)}$.
\item For any $q\in [2, \infty)$, we have $e^{\eta |\x|} \psi \in \sL^q(\R^2)$ and
$$
\| e^{\eta |\x|} \psi \|_{\sL^q(\R^2)} \leq C \|\psi \|_{\sL^q(\R^2)}.
$$
\end{itemize}
\end{corollary}

\begin{proof}
We give the proof for $k=0$. The proof for $k \geq 1$ is identical using Proposition~\ref{Agmon-k} below instead of Proposition~\ref{Agmon}.

It follows from Proposition~\ref{Agmon} and the diamagnetic inequality (Lemma~\ref{lem:diamagn}) that a minimizer $\psi$ satisfies $\| |\psi| \|_{\sH^1(\R^2)} \leq C \| \psi \|_{\sL^2(\R^2)}$. Therefore, we get by the Sobolev inequality \eqref{eq:Sobolev} that
$$
\| \psi \|_{\sL^q(\R^2)} \leq C \| \psi \|_{\sL^2(\R^2)},
$$
for any $q \in [2, \infty)$. On the other hand, we can use the H\"{o}lder inequality, followed by the previous inequality, to estimate
$$
\| \psi \|_{\sL^2(\R^2)}^2 \leq \| \psi \|_{\sL^q(\R^2)} \|\psi \|_{\sL^{q'}(\R^2)} \leq C \| \psi \|_{\sL^q(\R^2)} \|\psi \|_{\sL^2(\R^2)}.
$$
This proves the first parts of the corollary.

The exponential bound in $\sL^q$ now follows from Proposition~\ref{Agmon} and the diamagnetic inequality (Lemma~\ref{lem:diamagn}) and the Sobolev inequality \eqref{eq:Sobolev}.
\end{proof}

We will need the following lemma.
\begin{lemma}\label{lambdaNeu}
Let us consider $R>0$ and a family of smooth vector potentials $(\A_{n})_{n\geq 0}$ on $\overline{D(0,R)}$ such that $\B_{n}=\nabla\times\A_{n}\to+\infty$ uniformly on $\overline{D(0,R)}$.Then, the lowest eigenvalue $\lambda^\Neu_{1}(D(0,R),\A_{n})$ of the Neumann realization of the magnetic Laplacian $(-i\nabla+\A_{n})^2$ on $D(0,R)$ tends to $+\infty$.
\end{lemma}
\begin{proof}
We start by introducing an auxiliary operator.
Let $R/2 \leq r \leq R$ and consider the non-magnetic Laplace operator $-\Delta$ on the annulus $D(0,R) \setminus \overline{D(0,r)}$ with Dirichlet condition at $r$ and Neumann condition at $R$. If we let $\zeta(r)$ be the lowest eigenvalue of this operator, then it is a simple fact that $(R-r)^2 \zeta(r) \geq \delta_0 >0$ for some $\delta_0$ independent of $r$.

Let $\psi_n$ be an $\sL^2$-normalized ground state of the Neumann realization of the magnetic Laplacian $(-i\nabla+\A_{n})^2$ on $D(0,R)$. We let $q_n$ denote the quadratic form of $(-i\nabla+\A_{n})^2$.

Assume for contradiction that $\lambda^\Neu_{1}(D(0,R),\A_{n})$ remains bounded (along a subsequence).
Let $r<R$ and define $m_n(r)=\int_{D(0,r)} |\psi_n(x)|^2 \dx\x$. We start by proving that $m_n(r) \rightarrow 0$ as $n \rightarrow \infty$. In order to prove this, let us consider a partition of unity with $\chi_1^2 + \chi_2^2=1$, $\supp \chi_1 \subset D(0,r+\frac{R-r}{2})$, $\chi_1 = 1$ on $D(0,r)$ and such that $|\nabla\chi_{1}|^2+|\nabla\chi_{2}|^2\leq\frac{C}{(R-r)^2}$. By the IMS-formula \eqref{eq:IMS} we have
\begin{align}
\label{eq:IMS2-5}
q_n(\psi_n) \geq q_n( \chi_1 \psi_n) + q_n( \chi_2 \psi_n) -\frac{C}{|R-r|^2} \int_{\{r \leq |\x| \leq r+\frac{R-r}{2}\}} |\psi_n|^2\dx\x.
\end{align}
Since $ \chi_1 \psi_n$ has compact support in $D(0,R)$, we can estimate, using \eqref{eq:MagnUncertain},
\begin{align}
q_n( \chi_1 \psi_n) \geq \int \B_{n} |\chi_1 \psi_n|^2\dx\x \geq (\inf \B_{n}) m_n(r).
\end{align}
Since, by assumption, $q_n(\psi_n)$ is bounded, we conclude that $m_n(r) \rightarrow 0$.

But now we can reconsider \eqref{eq:IMS2-5} to get, with the diamagnetic inequality,
\begin{align}
q_n(\psi_n) \geq \zeta(r) \| \chi_2 \psi_n \|_2^2 - \frac{C m_{n}(r+\frac{R-r}{2})}{|R-r|^2}.
\end{align}
Since $m(r+\frac{R-r}{2}) \rightarrow 0$ and $ \| \chi_2 \psi_n \|_2^2 \rightarrow 1$ as $n \rightarrow \infty$, we find
\begin{align}
\liminf_{n \rightarrow \infty} q_n(\psi_n) \geq \zeta(r).
\end{align}
But $\zeta(r) \rightarrow \infty$ as $r \rightarrow R$, which is a contradiction.
\end{proof}

\begin{proposition}\label{Agmon-k}
For $k\geq 1$ and $p>2$, the infimum in \eqref{lambda0} is a minimum. Moreover, if $\psi$ is a minimizer, there exist $\eta, C>0$ such that 
\begin{align}
\int e^{2\eta |\x|} ( |\psi(\x)|^2 + |(-i\nabla+\A(\x))\psi)(\x)|^2)\dx x \leq C \| \psi \|_{\sL^2(\R^2)}^2. 
\end{align}
\end{proposition}

\begin{proof}
The existence of a minimizer is not a consequence of the results in \cite{EL89}. Nevertheless we will also use the concentration-compactness method. For simplicity of notation we will write $\A$ instead of $\A^{[k]}$ in this proof. Since $k \geq 1$ is fixed there is no room for confusion.

Let us consider a minimizing sequence $(u_{n})$, with $\|u_{n}\|_{\sL^p(\R^2)}=1$. We introduce the density measure $\mu_{n}=(|u_{n}|^2+|(-i\nabla+\A)u_{n}|^2)\dx\x$ whose total mass $\mu_{n}(\R^2)$ is bounded and we can assume that it converges to some $\mu>0$ up to the extraction of a subsequence. Indeed, if $\mu=0$, by the diamagnetic inequality, we would get that $(|u_{n}|)$ goes to $0$ in $\sH^1(\R^2)$ and thus in $\sL^p(\R^2)$.

Since $p>2$, as in \cite{EL89} and using \cite[Lemma 1]{Lions84}, we are easily reduced to the \enquote{tightness} case (see Appendix \ref{A}). In other words we may find a sequence $\x_{n}=(x_{n}, y_{n})$ such that
\begin{equation}\label{compactness}
\forall \eps>0,\quad\exists R>0, \quad \forall n\geq 1,\quad \mu_{n}(\complement D(\x_{n},R))\leq \eps.
\end{equation}
We introduce the translated function $\hat u_{n}(\x)=u_{n}(\x+\x_{n})$ and $\A_{n}(\x)=\A(\x+\x_{n})$.
Notice at this point that with our choice of vector potential $\A$ only depends on $\x = (x,y)$ through the first coordinate $x$.
We have
$$Q_{\A}(u_{n})=Q_{\A_{n}}(\hat u_{n})=\int_{\R^2} |D_{x}\hat u_{n}|^2+\left|\left(D_{y}+\frac{1}{k+1}(x-x_{n})^{k+1}\right)\hat u_{n}\right|^2 \dx\x.$$
From \eqref{compactness}, we get $\int_{\complement D(0,R)}|\hat u_{n}|^2\dx\x\leq\eps$, and we also have
$$\int_{D(0,R)} |D_{x}\hat u_{n}|^2+\left|\left(D_{y}+\frac{1}{k+1}(x-x_{n})^{k+1}\right)\hat u_{n}\right|^2 \dx\x\leq C,$$
for some $C$ independent of $n$.
By the min-max principle, we have
\begin{multline*}
\lambda_{1}^\Neu(D(0,R),\A_{n})\int_{D(0,R)}|\hat u_{n}|^2\dx\x\\
\leq\int_{D(0,R)} |D_{x}\hat u_{n}|^2+\left|\left(D_{y}+\frac{1}{k+1}(x-x_{n})^{k+1}\right)\hat u_{n}\right|^2 \dx\x.
\end{multline*}
If $|x_{n}|\to+\infty$, we get, by Lemma \ref{lambdaNeu} that $\lambda_{1}^\Neu(D(0,R),\A_{n})\to +\infty$ and thus there exists $N\geq 1$ such that, for all $n\geq N$,
$$\int_{D(0,R)}|\hat u_{n}|^2\dx\x\leq \eps.$$
We infer that $(u_{n})$ tends to $0$ in $\sL^2(\R^2)$. Since the diamagnetic inequality implies that $(|u_{n}|)_{n\geq 0}$ is bounded in $\sH^1$, we get a contradiction to the assumption that $\| u_n \|_p = 1$ by using a Sobolev embedding. Therefore we may assume that $x_{n}$ converges to some $x_{*}\in \R$. From this we infer that 
$(\nabla\hat u_{n})$ is bounded in $\sH^1(D(0,R))$ and since $\int_{\complement D(0,R)}|\hat u_{n}|^2\dx\x\leq\eps$, we can use the Rellich's Criterion \cite[Theorem XIII.65]{ReSi78} to see that we may assume that $\hat u_{n}$ converges in $\sL^2(\R^2)$ to some $\hat u_{*}$. In fact, the relative compactness is also verified in $\sL^q(\R^2)$ with $q\geq 2$, since the H\"older inequality provides
$$\int_{\complement D(0,R)} |\hat u_{n}|^q \dx\x\leq \left(\int_{\complement D(0,R)} |\hat u_{n}|^2 \dx\x\right)^{\frac{1}{2}}\left(\int_{\complement D(0,R)} |\hat u_{n}|^{2(q-1)} \dx\x\right)^{\frac{1}{2}}$$
and that, by diamagnetism,
$$\int_{\complement D(0,R)} |\hat u_{n}|^{2(q-1)} \dx\x\leq  \|\hat u_{n}\|_{\sL^{2q-2}(\R^2)}^{q-1}\leq C Q_{\A_{n}}(\hat u_{n})^{q-1}\leq \tilde C.$$
Therefore we may assume that $\hat u_{n}$ converges in $\sL^p(\R^2)$ so that we deduce, by translation invariance, $1=\|u_{n}\|_{\sL^p(\R^2)}=\|\hat u_{n}\|_{\sL^p(\R^2)}\to \|\hat u_{*}\|_{\sL^p(\R^2)}$ and thus $\|\hat u_{*}\|_{\sL^p(\R^2)}=1$. Moreover, up to extractions of subsequences and a diagonal argument, we can assume that $(\hat u_{n})$ converges to $\hat u_{*}$ weakly in $\sH^1_{\loc}(\R^2)$.

Then, we can conclude the proof. Indeed, we have, for all $R>0$ and $n\geq 1$,
$$Q_{\A}(u_{n})=Q_{\A_{n}}(\hat u_{n})\geq \int_{D(0,R)} |D_{x}\hat u_{n}|^2+\left|\left(D_{y}+\frac{1}{k+1}(x-x_{n})^{k+1}\right)\hat u_{n}\right|^2 \dx\x$$
so that, due to the weak convergence in $\sH^1_{\loc}(\R^2)$,
$$\liminf_{n\to+\infty}Q_{\A}(u_{n})\geq \int_{D(0,R)} |D_{x}\hat u_{*}|^2+\left|\left(D_{y}+\frac{1}{k+1}(x-x_{*})^{k+1}\right)\hat u_{*}\right|^2 \dx\x.$$
Therefore, 
$$\liminf_{n\to+\infty}Q_{\A}(u_{n})\geq Q_{\A_{*}}(\hat u_{*})=Q_{\A}(u_{*}),$$
where $u_{*}(x,y)=\hat u_{*}(x+x_{*},y)$, $\A_{*}(x,y) = (0, \frac{1}{k+1}(x-x_{*})^{k+1})$, and we also have $\|u_{*}\|_{\sL^p(\R^2)}=1$.
\end{proof}

To stress the difference between the linear ($p=2$) and the non-linear problem ($p>2$) we include the following simple result.

\begin{proposition}\label{Agmon-0}
For $k\geq 1$ and $p=2$, the infimum in \eqref{lambda0} is not a minimum.
\end{proposition}

\begin{proof}For $\alpha \in \R$ and $k \geq 1$, we define the `Montgomery operator' (or anharmonic oscillator) of order $k$,
\begin{equation*}
\Qk(\alpha) = D_{t}^2+\Bigl(\frac{t^{k+1}}{k+1}-\alpha\Bigr)^2,
\end{equation*}
as a self-adjoint operator in $\sL^2({\mathbb R})$.
Let $\eigone{\Qk(\alpha)}$ be its ground state eigenvalue.
By \cite[Theorem~1.3]{FouPer13}, for all $k \geq 1$, there exists a unique point $\alpha_0^{[k]} \in \R$ such that the function $\alpha \mapsto \eigone{\Qk(\alpha)}$ attains its minimum at $\alpha_0^{[k]}$. 
Also $\eigone{\Qk(\alpha)} \rightarrow \infty$ as $\alpha \rightarrow \infty$.
By partial Fourier transform in the $y$-coordinate and thanks to \cite[Theorem XIII.85]{ReSi78}, we get
$$
\lambda^{[k]}(p=2) = \eigone{\Qk(\alpha_0^{[k]})}.
$$
Suppose now by contradiction that $\psi$ is an $\sL^2$-normalized eigenfunction of the magnetic Laplacian $(-i\nabla +\A^{[k]})^2$ corresponding to $\lambda^{[k]}(p=2)$.
Let $\tilde{\psi}(x,\alpha) \in \sL^2(\R^2)$ be the partial Fourier transform of $\psi$ in the $y$ variable. In particular,
\begin{align}
\lambda^{[k]}(2) = \int_{\R} \left(\int_{\R} |D_{x}\tilde\psi|^2 + \left|\left(\alpha - \frac{x^{k+1}}{k+1}\right)\tilde\psi\right|^2 \dx x
\right) \dx\alpha.
\end{align}
By normalization and choosing $\delta >0$ small enough, we may assume that we have $\int_{\{ |\alpha - \alpha_0^{[k]}|\geq \delta \}}  |\tilde\psi(x,\alpha)|^2 \dx x\dx\alpha \geq 1/2$.
Using the continuity with respect to $\alpha$ and the uniqueness of the minimum, there exists $\eps >0$ such that 
\begin{align*}
\inf_{\{ |\alpha - \alpha_0^{[k]}|\geq \delta \}} \eigone{\Qk(\alpha)} \geq \eigone{\Qk(\alpha_0^{[k]})} + \eps,
\end{align*}
and so we get
\begin{multline*}
\lambda^{[k]}(2)\geq\int_{\R^2} \eigone{\Qk(\alpha)}|\tilde\psi(x,\alpha)|^2\dx x\dx\alpha\\
\geq (\eigone{\Qk(\alpha_0^{[k]})} + \eps) \int_{\{|\alpha - \alpha_0^{[k]}|\geq \delta\}}|\tilde\psi(x,\alpha)|^2 \dx x\dx\alpha+\eigone{\Qk(\alpha_0^{[k]})} \int_{\{|\alpha - \alpha_0^{[k]}|<\delta\}}|\tilde\psi(x,\alpha)|^2 \dx x\dx\alpha
\end{multline*}
and thus 
$$\lambda^{[k]}(2)\geq \eigone{\Qk(\alpha_0^{[k]})} + \frac{\eps}{2}.$$
This is a contradiction and finishes the proof.
\end{proof}

\section{Upper bounds}\label{Upper}
This section is devoted to the proof of the upper bounds in Theorems \ref{theo1} and \ref{theo2}.

\subsection{Non-vanishing magnetic field}
In this section we work under the assumptions of Theorem \ref{theo1}. Let us consider $v$ a minimizer associated with \eqref{lambda0} for $k=0$ and let
$$\psi(\x)=h^{-\frac{1}{p}}e^{i\frac{\phi(\x)}{h}}\chi(\x)v\left(\frac{\x-\x_{0}}{h^{\frac{1}{2}}}\right).$$
Here $\x_{0}$ denotes a point in $\Omega$ where the minimum of the magnetic field is obtained, $\chi \in C^{\infty}_{0}(\Omega)$, with $\chi \equiv 1$ in a neighborhood of $\x_0$, and $\phi$ is a real function such that $\tilde \A=\A+\nabla \phi$
satisfies in a fixed neighborhood of $\x_{0}$:
$$\left|\tilde \A(\x)-b_{0}\tilde\A^{[0]}(\x)\right|\leq C |\x-\x_{0}|^2,\qquad \tilde\A^{[0]}(\x)=\A^{[0]} (\x-\x_{0})$$
We have, by Corollary \ref{cor:AgmonL-p},
$$\int_{\Omega} |\psi(\x)|^p\dx \x=\int_{\R^2} \chi^p\left(\x_{0}+h^{\tfrac{1}{2}}\y\right)|v(\y)|^p\dx \y=\int_{\R^2} |v(\y)|^p \dx\y+\mathcal{O}(h^{\infty})\|v\|^p_{\sL^p(\R^2)}$$
and, with the \enquote{IMS} formula, 
\begin{multline*}
\mathcal{Q}_{h,\A}(\psi)=h^{-\frac{2}{p}}\int_{\Omega}  \chi^2\left(\x_{0}+h^{\tfrac{1}{2}}\y\right)\left|\left(-ih\nabla+\tilde\A\right)v\left(\frac{\x-\x_{0}}{h^{\frac{1}{2}}}\right)\right|^2\dx \x\\
+\mathcal{O}(h^{\infty})\|v\|_{\sL^2(\R^2)}^2
\end{multline*}
so that, for all $\eps>0$,
\begin{multline*}
 h^{\frac{2}{p}}\mathcal{Q}_{h,\A}(\psi)\leq(1+\eps)\int_{\Omega} \left|\left(-ih\nabla+b_{0}\tilde\A^{[0]}\right)v\left(\frac{\x-\x_{0}}{h^{\frac{1}{2}}}\right)\right|^2\dx \x\\
+(1+\eps^{-1})\int_{\Omega} \left|\left(\tilde\A-b_{0}\tilde\A^{[0]}\right)v\left(\frac{\x-\x_{0}}{h^{\frac{1}{2}}}\right)\right|^2\dx \x+\mathcal{O}(h^{\infty})\|v\|_{\sL^2(\R^2)}^2.
\end{multline*}
Due to the exponential decay of $v$ given in Proposition \ref{Agmon}, 
we have 
$$\int_{\R^2} |\y|^4 |v(\y)|^2\dx \y<+\infty,$$
and thus
\begin{multline*}
 h^{\frac{2}{p}}\mathcal{Q}_{h,\A}(\psi)\leq(1+\eps)h^2\int_{\R^2} \left|\left(-i\nabla+b_{0}\tilde\A^{[0]}\right)v(\y)\right|^2\dx \y\\
+C^2(1+\eps^{-1})h^3\int_{\R^2} |v(\y)|^2\dx \x+\mathcal{O}(h^{\infty})\|v\|_{\sL^2(\R^2)}^2.
\end{multline*}
We have by \eqref{eq:MagnUncertain},
$$\int_{\R^2} \left|\left(-i\nabla+b_{0}\tilde\A^{[0]}\right)v(\y)\right|^2\dx \y\geq b_{0}\int_{\R^2} |v(\y)|^2 \dx \y.$$
We deduce the upper bound:
\begin{align*}
 h^{\frac{2}{p}}\mathcal{Q}_{h,\A}(\psi)\leq\left((1+\eps)h^2+b_{0}^{-1}C^2(1+\eps^{-1})h^3\right)\int_{\R^2} \left|\left(-i\nabla+b_{0}\tilde\A^{[0]}\right)v(\y)\right|^2\dx \y.
\end{align*}
We take $\eps=h^{1/2}$ so that,
$$ h^{\frac{2}{p}}\lambda(\Omega,\A,p,h)\leq\left(h^2+Ch^{5/2}\right)\frac{\int_{\R^2} \left|\left(-i\nabla+b_{0}\tilde\A^{[0]}\right)v(\y)\right|^2\dx \y}{\left(\int_{\R^2} |v(\y)|^p \dx \y\right)^{\frac{2}{p}}}.$$
We get
$$\lambda(\Omega,\A,p,h)\leq h^{-\frac{2}{p}}\left(h^2+Ch^{5/2}\right)\lambda(1,b_{0}\tilde\A^{[0]},p).$$
By homogeneity and gauge invariance, we have
$$\lambda(\R^2, b_{0}\tilde\A^{[0]},p,1)=b_{0}^{\frac{2}{p}}\lambda(\R^2,\A^{[0]},p,1).$$
We infer the upper bound
\begin{align}\label{eq:Upper_Theorem1.4}
\lambda(\Omega,\A,p,h)\leq h^{-\frac{2}{p}} \left(b_{0}^{\frac{2}{p}}h^2\lambda(\R^2,\A^{[0]},p,1)+Ch^{\frac{5}{2}}\right),
\end{align}
So the upper bound of Theorem~\ref{theo1} is proved.

\subsection{Vanishing magnetic field}
Let us now work under the assumptions of Theorem \ref{theo2}. We can define the standard tubular coordinates in a neighborhood of a point $\x_{0}\in\Gamma$ which minimizes the normal derivative of $\B$, $\Gamma\ni x\mapsto \partial_{\n,\Gamma}\B(x)$. These coordinates are defined through the local diffeomorphism $\Phi$ : $(s,t)\mapsto c(s)+t\n(c(s))=\x$ where $c$ is a parametrization of $\Gamma$ such that $|c'(s)|=1$ and $\n(c(s))$ is the inward pointing normal of $\Gamma$ at $c(s)$, that is $\det(c'(s),\n(c(s)))=1$. We may assume that $\Phi(0,0)=\x_{0}$. For further details, we refer to \cite[Appendix F]{FouHel10}. In these new coordinates the quadratic form becomes, for functions $\psi$ supported near $\x_{0}$,
\begin{align}\label{eq:locCOord}
\mathcal{Q}_{h,\A}(\psi)=\widetilde{\mathcal{Q}}_{h,\A}(\tilde\psi)=\int \left\{|hD_{t}\tilde\psi|^2+(1-tk(s))^{-2}|(hD_{s}+\tilde A)\tilde\psi|^2 \right\} (1-tk(s))\dx s\dx t,
\end{align}
with $\tilde\psi(s,t)=e^{i\varphi(s,t)/h}\psi(\Phi(s,t))$, where $\varphi$ corresponds to a local change of gauge. Moreover we have let $k(s) = \gamma''(s) \cdot \n(\gamma(s))$, and 
$$\tilde A(s,t)=-\int_{0}^t (1-uk(s))\tilde{\B}(s,u)\dx u,\qquad \mbox{ with }\quad\tilde{\B}(s,t)=\B(\Phi(s,t)).$$
Note that we also have, as soon as $\psi$ is supported near $\Gamma$,
$$\int_{\Omega} |\psi|^p \dx\x=\int |\tilde\psi|^p (1-tk(s))\dx s\dx t.$$
We may write the following Taylor estimate
$$\left|\tilde A(s,t)+\gamma_{0}\frac{t^2}{2}\right|\leq C(|s|^3+|t|^3).$$
Let us consider $w$ the complex conjugate of a minimizer associated with \eqref{lambda0} for $k=1$, normalized in $\sL^p(\R^2)$ and let
$$\tilde\psi(s,t)=h^{-\frac{2}{3p}}\gamma^{\frac{2}{3p}}_{0}\chi(s,t)w\left(\gamma_{0}^{\frac{1}{3}} \frac{s}{h^{\frac{1}{3}} }, \gamma_{0}^{\frac{1}{3}}\frac{t}{h^{\frac{1}{3} }}\right),
$$
where $\chi \in C_0^{\infty}(\R^2)$, $\chi \equiv 1$ near $0$, with $\supp \chi$ sufficiently small.

We have, using the exponential decay of $w$,
$$\int |\tilde\psi(s,t)|^p (1-tk(s))\dx s\dx t=\int |w(\sigma,\tau)|^p (1-\tau \gamma_{0}^{-\frac{1}{3}}h^{\frac{1}{3}} k(h^{\frac{1}{3}}\gamma_{0}^{-\frac{1}{3}}\sigma))\dx \sigma\dx \tau+\mathcal{O}(h^\infty)\|w\|_{\sL^p(\R^2)}^p$$
so that,
$$\int |\tilde\psi(s,t)|^p (1-tk(s))\dx s\dx t\geq (1-Ch^{\frac{1}{3}})\int |w(\sigma,\tau)|^p \dx\sigma \dx\tau.$$
Thanks to support considerations, we get
\begin{multline*}
\int \left\{|hD_{t}\tilde\psi|^2+(1-tk(s))^{-2}|(hD_{s}+\tilde A)\tilde\psi|^2 \right\} (1-tk(s))\dx s\dx t\\
\leq \int\left\{|hD_{t}\tilde\psi|^2+|(hD_{s}+\tilde A)\tilde\psi|^2 \right\} \dx s\dx t+C\int |t|\left\{|hD_{t}\tilde\psi|^2+|(hD_{s}+\tilde A)\tilde\psi|^2 \right\} \dx s\dx t.
\end{multline*}
With the exponential decay of $w$, we have
$$\int |t|\left\{|hD_{t}\tilde\psi|^2+|(hD_{s}+\tilde A)\tilde\psi|^2 \right\} \dx s\dx t\leq Ch^{\frac{5}{3}} h^{-\frac{4}{3p}}h^{\frac{2}{3}}. $$
In the same way we get
\begin{multline*}
\int\left\{|hD_{t}\tilde\psi|^2+|(hD_{s}+\tilde A)\tilde\psi|^2 \right\} \dx s\dx t\\
\leq\int\left\{|hD_{t}\tilde\psi|^2+\left|\left(hD_{s}-\gamma_{0}\frac{t^2}{2}\right)\tilde\psi\right|^2 \right\} \dx s\dx t+Ch^{\frac{5}{3}} h^{-\frac{4}{3p}}h^{\frac{2}{3}}
\end{multline*}
and 
$$\int\left\{|hD_{t}\tilde\psi|^2+\left|\left(hD_{s}-\gamma_{0}\frac{t^2}{2}\right)\tilde\psi\right|^2 \right\} \dx s\dx t=\gamma_{0}^{\frac{4}{3p}}h^{\frac{4}{3}} h^{\frac{2}{3}}h^{-\frac{4}{3p}}\lambda^{[1]}(p )+\mathcal{O}(h^{\infty}).$$
We deduce
$$\frac{\mathcal{Q}_{h,\A}(\psi)}{\|\psi\|^2_{\sL^p(\R^2)}}\leq(1+Ch^{\frac{1}{3}})\gamma_{0}^{\frac{4}{3p}}h^2 h^{-\frac{4}{3p}}\lambda^{[1]}(p )$$
and the conclusion immediately follows.
\section{Lower bounds}\label{Lower}
We are now interested in the lower bounds in Theorems \ref{theo1} and \ref{theo2}.

\subsection{Quadratic partition of unity and reconstruction of $\sL^p$-norm}
Let us introduce a quadratic partition of unity \enquote{with small interaction supports}. In the following we will for notational convenience use the $\infty$-norm on $\R^2$, explicitly $|\x |_{\infty} = \max( |x_1|, |x_2|)$.
\begin{lemma}\label{lemma-partition}
Let us consider $E=\{(\alpha,\rho,h,\ell)\in(\R_{+})^3\times\Z^2 : \alpha\geq \rho\}$. There exists a family of smooth cutoff functions $(\chi^{[\ell]}_{\alpha,\rho,h})_{(\alpha,\rho,h,\ell)\in E}$ on $\R^2$ such that $0\leq \chi^{[\ell]}_{\alpha,\rho,h}\leq 1$, 
\begin{align*}
\chi^{[\ell]}_{\alpha,\rho,h}&=1, \quad \text{  on }\quad |\x-(2h^{\rho}+h^{\alpha})\ell|_{\infty}\leq h^{\rho},\\
\chi_{\alpha,\rho,h}&=0, \quad\text{ on }\quad |\x-h^{\rho}\ell|_{\infty}\geq h^{\rho}+h^{\alpha},
\end{align*} and such that
$$\sum_{\ell\in\Z^2}\left(\chi_{\alpha,\rho,h}^{[\ell]}\right)^2=1.$$
Moreover there exists $D>0$ such that, for all $h>0$,
\begin{equation}\label{partition-remainder}
\sum_{\ell\in\Z^2} |\nabla\chi_{\alpha,\rho,h}^{[\ell]}|^2\leq Dh^{-2\alpha}.
\end{equation}
\end{lemma}

\begin{proof}
Let us consider $F=\{(\alpha,\rho,h)\in(\R_{+})^3 : \alpha\geq \rho\}$.
There exists a family of smooth cutoff functions of one real variable $(\chi_{\alpha,\rho,h})_{(\alpha,\rho,h)\in F}$ such that $0\leq \chi_{\alpha,\rho,h}\leq 1$, $\chi_{\alpha,\rho,h}=1$ on $|x|\leq h^{\rho} + \frac{1}{2} h^{\alpha}$ and $\chi_{\alpha,\rho,h}=0$ on $|x|\geq h^{\rho}+h^{\alpha}$, and such that for all $(\alpha,\rho)$ with $\alpha\geq\rho>0$, there exists $C>0$ such that for all $h>0$, $|\nabla\chi_{\alpha,\rho,h}|\leq Ch^{-\alpha}$.
Then, we define :
$$S_{\alpha,\rho,h}(x)=\sum_{\ell\in\Z^2}\chi^2_{\alpha,\rho, h}\big(x_{1}-(2h^{\rho}+h^{\alpha})\ell_{1}\big)\chi^2_{\alpha,\rho, h}\big(x_{2}-(2h^{\rho}+h^{\alpha})\ell_{2}\big),$$
and we have 
$$\forall x\in\R^2,\qquad 1\leq S_{\alpha,\rho,h}(x)\leq 4.$$
We let
$$\chi_{\alpha,\rho,h}^{[\ell]}(x)=\frac{\chi_{\alpha,\rho, h}(x_{1}-(2h^{\rho}+h^{\alpha})\ell_{1})\chi_{\alpha,\rho, h}(x_{2}-(2h^{\rho}+h^{\alpha})\ell_{2})}{\sqrt{S_{\alpha,\rho,h}(x)}},$$
which satisfies the wished estimates by standard arguments.
\end{proof}

Given a grid and a non-negative and integrable function $f$, the following lemma states that, up to a translation of the grid, the mass of $f$ carried by a slightly thickened grid is controlled by a slight fraction of the total mass of $f$.
\begin{lemma}\label{translation}
For $r>0$ and $\delta>0$, we define the grid $\Lambda_{r}=((r\Z)\times\R)\cup(\R\times (r\Z))$ and the thickened grid
$$\Lambda_{r,\delta}=\{x\in\R^2 : \dist(x,\Lambda_{r})\leq \delta\}.$$
Let us consider a non-negative function $f$ belonging to $\sL^1(\R^2)$. Then there exists $\tau(r,\delta,f)=\tau\in\R^2$ such that :
$$\int_{\Lambda_{r,\delta}+\tau} f(x)\dx x\leq\frac{3\delta}{r+2\delta}\int_{\R^2} f(x)\dx x.$$
\end{lemma}

\begin{proof}
We let $\e=\frac{1}{\sqrt{2}}(1,1)$. We notice that
$$\sum_{j=0}^{\lfloor \frac{r}{2\delta}\rfloor+1}\int_{\Lambda_{r,\delta}+j\delta \e} f(x) \dx x=\int_{\R^2} g_{r,\delta}(x) f(x)\dx x,\qquad\mbox{ with } g_{r,\delta}(x)=\sum_{j=0}^{\lfloor \frac{r}{2\delta}\rfloor+1} \mathds{1}_{\Lambda_{\delta}+j\delta \e}(x).$$
We have, for almost all $x$, $g_{r,\delta}(x)\leq 3$, so that we get
$$\sum_{j=0}^{\lfloor \frac{r}{2\delta}\rfloor+1}\int_{\Lambda_{r,\delta}+j\delta \e} f(x) \dx x\leq 3\int_{\R^2} f(x)\dx x.$$
Therefore, there exists $j\in\left\{0,\ldots,\lfloor \frac{r}{\delta}\rfloor+1\right\}$, such that
$$\int_{\Lambda_{r,\delta}+j\delta \e} f(x) \dx x\leq \frac{3}{\lfloor \frac{r}{2\delta}\rfloor+2}\int_{\R^2} f(x)$$
and the conclusion easily follows.
\end{proof}
We can now establish the following lemma which permits to recover the total $\sL^p$-norm from the local contributions defined by the quadratic partition of unity.
\begin{lemma}\label{Lp-partition}
Let $p \geq 2$.
Let us consider the partition of unity $(\chi_{\alpha,\rho,h}^{[\ell]})$ defined in Lemma \ref{lemma-partition}, with $\alpha>\rho>0$. There exist $C>0$ and $h_{0}>0$ such that for all $\psi\in \sL^p(\Omega)$ and $h\in(0,h_{0})$, there exists $\tau_{\alpha,\rho,h,\psi}=\tau\in\R^2$ such that
$$\sum_{\ell} \int_{\Omega} |\tilde\chi_{\alpha,\rho,h}^{[\ell]}\psi(\x)|^p \dx \x\leq\int_{\Omega}|\psi(\x)|^p\dx \x\leq (1+Ch^{\alpha-\rho})\sum_{\ell} \int_{\Omega} |\tilde\chi_{\alpha,\rho,h}^{[\ell]}\psi(\x)|^p \dx \x,$$
with $\tilde\chi_{\alpha,\rho,h}^{[\ell]}(\x)=\tilde\chi_{\alpha,\rho,h}^{[\ell]}(\x-\tau)$.
Moreover, the translated partition $(\tilde\chi_{\alpha,\rho,h}^{[\ell]})$ still satisfies \eqref{partition-remainder}.
\end{lemma}

\begin{proof}
The first inequality is obvious since the cutoff functions are bounded by $1$ and their squares sum to unity. For the second inequality, we write, for any translation $\tau$,
$$\int_{\Omega}|\psi(\x)|^p\dx \x=\sum_{\ell}\int_{\Omega}\left(\tilde\chi_{\alpha,\rho,\ell}^{[\ell]}\right)^p|\psi(\x)|^p\dx \x+\int_{\Omega}\varphi_{\alpha,\rho}(\x)|\psi(\x)|^p\dx \x,$$
where
$$\varphi_{\alpha,\rho}=\sum_{\ell}\left(\left(\tilde\chi_{\alpha,\rho,\ell}^{[\ell]}\right)^2-\left(\tilde\chi_{\alpha,\rho,\ell}^{[\ell]}\right)^p\right).$$
The smooth function $\varphi_{\alpha,\rho}$ is supported on $\tau+\Lambda_{h^\rho+\frac{1}{2}h^{\alpha},2h^{\alpha}}$ and 
$$\int_{\Omega}\varphi_{\alpha,\rho}(\x)|\psi(\x)|^p\dx \x\leq \int_{\tau+\Lambda_{h^\rho+\frac{1}{2}h^{\alpha},2h^{\alpha}}}f(\x)\dx \x,$$
where $f(\x)=|\psi(\x)|^p$ for $\x\in\Omega$ and $f(\x)=0$ elsewhere. Thus, by Lemma \ref{translation}, we find $\tau$ such that
$$\int_{\Omega}\varphi_{\alpha,\rho}(\x)|\psi(\x)|^p\dx \x\leq Ch^{\alpha-\rho}\int_{\R^2} |\psi(\x)|^p\dx\x$$
and the conclusion easily follows. 
\end{proof}

\subsection{Lower bound: non-vanishing magnetic field}
This section is devoted to the proof of the lower bound in Theorem \ref{theo1}.
\subsubsection{A lower bound for the eigenvalue}
Let us consider $\psi\in\Dom(\mathcal{Q}_{h,\A})$. With the \enquote{IMS} formula associated with the partition of unity $(\tilde\chi_{\alpha,\rho,h}^{[\ell]})$ that is adapted to $\psi$ (see Lemma~\ref{Lp-partition}), we infer
$$\mathcal{Q}_{h,\A}(\psi)=\sum_{\ell}\mathcal{Q}_{h,\A}(\tilde\chi_{\alpha,\rho,h}^{[\ell]}\psi) -h^2\sum_{\ell}\|\nabla\tilde\chi_{\alpha,\rho,h}^{[\ell]}\psi\|_{\sL^2(\Omega)}^2.$$
We have
\begin{equation}\label{IMS-lb}
\mathcal{Q}_{h,\A}(\psi)\geq\sum_{\ell}\left(\mathcal{Q}_{h,\A}(\tilde\chi_{\alpha,\rho,h}^{[\ell]}\psi) -Dh^{2-2\alpha}\|\tilde\chi_{\alpha,\rho,h}^{[\ell]}\psi\|_{\sL^2(\Omega)}^2\right).
\end{equation}
By the min-max principle, we get
\begin{equation}\label{relative-control}
\lambda(\Omega,\A, 2, h)\|\tilde\chi_{\alpha,\rho,h}^{[\ell]}\psi\|_{\sL^2(\Omega)}^2\leq \mathcal{Q}_{h,\A}(\tilde\chi_{\alpha,\rho,h}^{[\ell]}\psi)
\end{equation}
and we recall that (see \cite[Theorem 1.1]{HelMo01})
\begin{equation}\label{HM}
\lambda(\Omega,\A, 2, h)=b_{0}h+\mathcal{O}(h^{\frac{3}{2}})
\end{equation}
so that
$$\mathcal{Q}_{h,\A}(\psi)\geq(1-Dh^{1-2\alpha})\sum_{\ell}\mathcal{Q}_{h,\A}(\tilde\chi_{\alpha,\rho,h}^{[\ell]}\psi).$$
Then, we bound the local energies from below. Thanks to support considerations (recall that $\alpha \geq \rho$), we have, modulo a local change of gauge $e^{i \phi^{[\ell]}/h}$, 
$$\mathcal{Q}_{h,\A}(\tilde\chi_{\alpha,\rho,h}^{[\ell]}\psi) \geq (1-\eps)\mathcal{Q}_{h,b_{j}\A^{[0]}}(e^{i \phi^{[\ell]}/h}\tilde\chi_{\alpha,\rho,h}^{[\ell]}\psi) -C\eps^{-1}h^{4\rho}\|\tilde\chi_{\alpha,\rho,h}^{[\ell]}\psi\|_{\sL^2(\Omega)}^2$$
so that it follows, by using again \eqref{relative-control},
$$\mathcal{Q}_{h,\A}(\tilde\chi_{\alpha,\rho,h}^{[\ell]}\psi) \geq (1-\eps-C\eps^{-1}h^{4\rho-1})\mathcal{Q}_{h,b_{j}\A^{[0]}}(e^{i \phi^{[\ell]}/h}\tilde\chi_{\alpha,\rho,h}^{[\ell]}\psi).$$
We take $\eps= h^{2\rho-\frac{1}{2}}$ and we deduce
\begin{equation}\label{lb}
\mathcal{Q}_{h,\A}(\psi)\geq (1-Dh^{1-2\alpha}-Ch^{2\rho-\frac{1}{2}})\sum_{\ell}b^{2/p}_{\ell}h^2h^{-2/p}\lambda^{[0]}( p)\|\tilde\chi_{\alpha,\rho,h}^{[\ell]}\psi\|_{\sL^p(\Omega)}^2
\end{equation}
so that
$$\mathcal{Q}_{h,\A}(\psi)\geq (1-Dh^{1-2\alpha}-Ch^{2\rho-\frac{1}{2}})b^{2/p}_{0}h^2h^{-2/p}\lambda^{[0]}( p)\sum_{\ell}\|\tilde\chi_{\alpha,\rho,h}^{[\ell]}\psi\|_{\sL^p(\Omega)}^2$$
Since $p\geq 2$, we have
$$\sum_{\ell}\|\tilde\chi_{\alpha,\rho,h}^{[\ell]}\psi\|_{\sL^p(\Omega)}^2\geq \left(\sum_{\ell}\int_{\Omega} |\tilde\chi_{\alpha,\rho,h}^{[\ell]}\psi|^p \dx x\right)^{\frac{2}{p}}.$$
Using Lemma \ref{Lp-partition}, we infer
\begin{equation}\label{recol-Lp}
\sum_{\ell}\|\tilde\chi_{\alpha,\rho,h}^{[\ell]}\psi\|_{\sL^p(\Omega)}^2 \geq (1-\tilde Ch^{\alpha-\rho})\|\psi\|^2_{\sL^p(\Omega)}.
\end{equation}
Finally, we get
$$\mathcal{Q}_{h,\A}(\psi)\geq (1-Dh^{1-2\alpha}-Ch^{2\rho-\frac{1}{2}})(1-\tilde Ch^{\alpha-\rho})b^{2/p}_{0}h^2h^{-2/p}\lambda^{[0]}(p )\|\psi\|^2_{\sL^p(\Omega)}.$$
Optimizing the remainders, we choose $1-2\alpha=2\rho-\frac{1}{2}=\alpha-\rho$ so that $\rho=\frac{5}{16}$ and $\alpha=\frac{7}{16}$ and
\begin{align}\label{eq:Lower_Theorem1.4}
\mathcal{Q}_{h,\A}(\psi)\geq (1-Ch^{\frac{1}{8}})b^{2/p}_{0}h^2h^{-2/p}\lambda^{[0]}( p)\|\psi\|^2_{\sL^p(\Omega)}.
\end{align}

This gives the lower bound needed for Theorem~\ref{theo1}. Combined with \eqref{eq:Upper_Theorem1.4}, \eqref{eq:Lower_Theorem1.4} yields the proof of Theorem~\ref{theo1}.

\subsubsection{A direct application to the localization}
The following proposition provides a rough (but quantitative) localization in $\sL^p$-norm of the minimizers near the minimum of the magnetic field.
\begin{proposition}\label{rough-loc}
Let us assume that $\B_{|\overline{\Omega}}$ admits a unique minimum attained at $0\in\Omega$. For all $\eps>0$, there exist $h_{0}>0$, $C>0$ such that for all $h\in(0,h_{0})$ and $\psi$ minimizer,
\begin{equation}\label{loc-Lp}
\|\psi\|_{\sL^p(\complement D(0,2\eps))}\leq Ch^{\frac{1}{8p}}\|\psi\|_{\sL^p(\Omega)}.
\end{equation}
In the case when the minimum is non-degenerate, this can be improved to
$$\|\psi\|_{\sL^p(\complement D(0,2h^{\tilde\rho}))}\leq Ch^{\left(\frac{1}{8}-2\tilde\rho\right)\frac{1}{p}}\|\psi\|_{\sL^p(\Omega)},$$
where $\tilde\rho<\frac{1}{16}$.
\end{proposition}
\begin{proof}
We apply \eqref{lb} to a minimizer $\psi$ and we get, with choices of $\rho$ and $\alpha$ given in the previous section:
$$\lambda(\Omega,\A,p,h)\|\psi\|^2_{\sL^p(\Omega)}\geq (1-Ch^{\frac{1}{8}})\sum_{\ell}b^{2/p}_{\ell}h^2h^{-2/p}\lambda^{[0]}( p)\|\tilde\chi_{\alpha,\rho,h}^{[\ell]}\psi\|_{\sL^p(\Omega)}^2.$$
With the upper bound of Theorem \ref{theo1} and \eqref{recol-Lp}, we get
\begin{align}\label{eq:4.9}
\sum_{\ell}\left\{b^{2/p}_{\ell}-b^{2/p}_{0}\right\}\|\tilde\chi_{\alpha,\rho,h}^{[\ell]}\psi\|_{\sL^p(\Omega)}^2\leq Ch^{\frac{1}{8}}\|\psi\|_{\sL^p(\Omega)}^2.
\end{align}
Let us introduce $K_{1}(h,\eps)=\{\ell : \x_{\ell}\in D(0,\eps)\}$. From \eqref{eq:4.9} and the uniqueness of the minimum, we have (for some $\eta>0$ and all $h$ sufficiently small), 
\begin{align}\label{eq:4.10}
\eta \sum_{\ell\notin K_{1}(h,\eps)} \|\tilde\chi_{\alpha,\rho,h}^{[\ell]}\psi\|_{\sL^p(\Omega)}^2\leq Ch^{\frac{1}{8}}\|\psi\|_{\sL^p(\Omega)}^2,
\end{align}
and (by concavity),
$$\sum_{\ell\notin K_{1}(h,\eps)} \|\tilde\chi_{\alpha,\rho,h}^{[\ell]}\psi\|_{\sL^p(\Omega)}^2\geq \left(\sum_{\ell\notin K_{1}(h,\eps)}\int_{\Omega}|\tilde\chi_{\alpha,\rho,h}^{[\ell]}\psi|^p\dx\x\right)^{\frac{2}{p}}.$$
so that
\begin{align}\label{eq:4.11}
\sum_{\ell\notin K_{1}(h,\eps)}\int_{\Omega}|\tilde\chi_{\alpha,\rho,h}^{[\ell]}\psi|^p\dx\x\leq Ch^{\frac{p}{16}}\|\psi\|_{\sL^p(\Omega)}^p.
\end{align}
We get, using Lemma~\ref{Lp-partition} in the last step,
\begin{align}\label{eq:4.12}
&\sum_{\ell\notin K_{1}(h,\eps)}\int_{\Omega}|\tilde\chi_{\alpha,\rho,h}^{[\ell]}\psi|^p\dx\x \nonumber\\
&=\sum_{\ell\notin K_{1}(h,\eps)}\int_{\Omega}|\tilde\chi_{\alpha,\rho,h}^{[\ell]}|^2|\psi|^p\dx\x+\sum_{\ell\notin K_{1}(h,\eps)}\int_{\Omega}(|\tilde\chi_{\alpha,\rho,h}^{[\ell]}|^p|\psi|^p-|\tilde\chi_{\alpha,\rho,h}^{[\ell]}|^2|\psi|^p)\dx\x \nonumber\\
&\geq\sum_{\ell\notin K_{1}(h,\eps)}\int_{\Omega}|\tilde\chi_{\alpha,\rho,h}^{[\ell]}|^2|\psi|^p\dx\x-Ch^{\frac{1}{8}}\|\psi\|_{\sL^p(\Omega)}^p.
\end{align}
We infer \eqref{loc-Lp} upon inserting \eqref{eq:4.12} in \eqref{eq:4.11}.

If the minimum of $\B$ is non-degenerate, i.e. the Hessian of $\B$ is strictly positive, \eqref{eq:4.10} improves to
$$\sum_{\ell\notin K_{1}(h,h^{\tilde\rho})} \|\tilde\chi_{\alpha,\rho,h}^{[\ell]}\psi\|_{\sL^p(\Omega)}^2\leq Ch^{\frac{1}{8}-\frac{4\tilde\rho}{p}}\|\psi\|_{\sL^p(\Omega)}^2,$$
with $\tilde\rho< \rho =\frac{5}{16}$ and we get the desired improvement by the same arguments. 
\end{proof}

\subsubsection{Proof of Theorem \ref{loc}}
By Proposition \ref{rough-loc}, there exists $\gamma>0$ such that 
\begin{align}\label{eq:Loc-L-p}
\|\psi\|^p_{\sL^p(\complement D(0,2\eps))}\leq Ch^{\gamma}\|\psi\|^p_{\sL^p(\Omega)}.
\end{align}
We assume that $\|\psi\|_{\sL^p(\Omega)}=1$. Then, we have the following estimate of the non-linear electric potential, for $v\in\Dom(\mathcal{Q}_{h,\A})$ and supported away from the ball $D(0,2\eps)$,
\begin{multline*}
\lambda(\Omega,\A,p,h)\int_{\Omega} |\psi|^{p-2}|v|^2\dx \x\leq \lambda(\Omega,\A,p,h)\|v\|_{\sL^p}^2\left(\int_{\complement D(0,2\eps)}|\psi|^p\dx\x\right)^{\frac{p-2}{p}}\\\leq C\lambda(\Omega,\A,p,h) h^{\gamma\frac{p-2}{p}}\|v\|_{\sL^p(\Omega)}^2,
\end{multline*}
where we used the H\"older inequality and \eqref{eq:Loc-L-p}. We now apply \eqref{re-Sob} to the extension by $0$ of $v$ to get
$$\|v\|^2_{\sL^p(\R^2)}\leq C(\eps^{1-\frac{p}{2}}\|v\|^2_{\sL^2(\R^2)}+\eps\|\nabla |v|\|^2_{\sL^2(\R^2)})$$
and the diamagnetic inequality implies
$$\|v\|^2_{\sL^p(\Omega)}\leq C(\eps^{1-\frac{p}{2}}\|v\|^2_{\sL^2(\Omega)}+\eps h^{-2}\mathcal{Q}_{h,\A}(v))$$
so that
$$\lambda(\Omega,\A,p,h)\int_{\Omega} |\psi|^{p-2}|v|^2\dx \x\leq  Ch^{2-\frac{2}{p}} h^{\gamma\frac{p-2}{p}}(\eps^{1-\frac{p}{2}}\|v\|^2_{\sL^2(\Omega)}+\eps h^{-2}\mathcal{Q}_{h,\A}(v)).$$
Let us now choose an appropriate $\eps$. We would like to have
$$h^{1-\frac{2}{p}} h^{\gamma\frac{p-2}{p}}\eps^{1-\frac{p}{2}}\ll 1,\quad h^{2-\frac{2}{p}} h^{\gamma\frac{p-2}{p}}\eps h^{-2}\ll 1$$
We choose $\eps=\delta h^{\frac{2}{p}+\gamma\left(-1+\frac{2}{p}\right)}$ with $\delta>0$ arbitrarily small. We get
$$h^{1-\frac{2}{p}} h^{\gamma\frac{p-2}{p}}\eps^{1-\frac{p}{2}}=\delta^{1-\frac{p}{2}}h^{1-\frac{2}{p}+\gamma\frac{p-2}{p}+\left(1-\frac{p}{2}\right)\left\{\frac{2}{p}+\gamma\left(-1+\frac{2}{p}\right)\right\}}=\delta^{1-\frac{p}{2}}h^{\frac{\gamma}{2}(p-2)}.$$
We could even choose a very small power of $h$ for $\delta$. We infer that, for $v$ in the domain of the magnetic Laplacian and supported in $\complement D(0,2\eps)$,
\begin{equation}\label{relative-bound}
\lambda(\Omega,\A,p,h)\int_{\Omega} |\psi|^{p-2}|v|^2\dx \x\leq C\left(h\delta^{1-\frac{p}{2}}h^{\frac{\gamma}{2}(p-2)}\|v\|_{\sL^2(\Omega)}^2+\delta\mathcal{Q}_{h,\A}(v)\right).
\end{equation}
Let us now establish our Agmon estimates. We recall the equation
$$\mathcal{L}_{h,\A,V_{h}}\psi=(-ih\nabla+\A)^2\psi+V_{h}\psi=0,\qquad\mbox{ with }\quad V_{h}=-\lambda(\Omega,\A,p,h)|\psi|^{p-2}.$$
We consider a smooth cutoff function $\chi$ such that $\chi=0$ in $D(0,2\eps)$, $0\leq \chi\leq 1$ and $\chi = 1$ on $\complement D(0,4\eps)$. We write the IMS formula and get
$$\mathcal{Q}_{h,\A,V_{h}}(e^{\chi h^{-\rho}|\x|}\psi)-Ch^{2-2\rho}\|e^{\chi h^{-\rho}|\x|}\psi\|_{\sL^2(\Omega)}^2\leq 0.$$
Then, we introduce a quadratic partition of unity
$$\chi_{1}^2+\chi_{2}^2=1,$$
such that $\chi_{2}$ is supported in $\complement D(0,2\eps)$. With the \enquote{IMS} formula, we deduce
\begin{multline*}
\mathcal{Q}_{h,\A,V_{h}}(\chi_{1}e^{\chi h^{-\rho}|\x|}\psi)+\mathcal{Q}_{h,\A,V_{h}}(\chi_{2}e^{\chi h^{-\rho}|\x|}\psi)\\
-\tilde Ch^{2-2\rho}\|\chi_{1}e^{\chi h^{-\rho}|\x|}\psi\|_{\sL^2(\Omega)}^2-\tilde Ch^{2-2\rho}\|\chi_{2}e^{\chi h^{-\rho}|\x|}\psi\|_{\sL^2(\Omega)}^2\leq 0.
\end{multline*}
Then, with \eqref{relative-bound}, we have
\begin{multline*}
\mathcal{Q}_{h,\A,V_{h}}(\chi_{2}e^{\chi h^{-\rho}|\x|}\psi)-\tilde Ch^{2-2\rho}\|\chi_{2}e^{\chi h^{-\rho}|\x|}\psi\|_{\sL^2(\Omega)}^2\\
\geq (1-C\delta)\mathcal{Q}_{h,\A}(\chi_{2}e^{\chi h^{-\rho}|\x|}\psi)-h\delta^{1-\frac{p}{2}}h^{\frac{\gamma}{2}(p-2)}\|\chi_{2}e^{\chi h^{-\rho}|\x|}\psi\|_{\sL^2(\Omega)}^2
\end{multline*}
Let us recall that
$$\mathcal{Q}_{h,\A}(\chi_{2}e^{\chi h^{-\rho}|\x|}\psi)\geq hb_{0}\|\chi_{2}e^{\chi h^{-\rho}|\x|}\psi\|_{\sL^2(\Omega)}^2$$
and we deduce
\begin{multline*}
\mathcal{Q}_{h,\A,V_{h}}(\chi_{2}e^{\chi h^{-\rho}|\x|}\psi)-\tilde Ch^{2-2\rho}\|\chi_{2}e^{\chi h^{-\rho}|\x|}\psi\|_{\sL^2(\Omega)}^2\\
\geq \left((1-C\delta)hb_{0}-\tilde Ch^{2-2\rho}\right)\|\chi_{2}e^{\chi h^{-\rho}|\x|}\psi\|_{\sL^2(\Omega)}^2\geq \eta h\|\chi_{2}e^{\chi h^{-\rho}|\x|}\psi\|_{\sL^2(\Omega)}^2,
\end{multline*}
as soon as $\rho\in\left(0,\frac{1}{2}\right)$ and for $h$ small enough. By support considerations, we have
$$\mathcal{Q}_{h,\A,V_{h}}(\chi_{1}e^{\chi h^{-\rho}|\x|}\psi)-\tilde Ch^{2-2\rho}\|\chi_{1}e^{\chi h^{-\rho}|\x|}\psi\|_{\sL^2(\Omega)}^2=\mathcal{Q}_{h,\A,V_{h}}(\chi_{1}\psi)-\tilde Ch^{2-2\rho}\|\chi_{1}\psi\|_{\sL^2(\Omega)}^2$$
and then
$$\mathcal{Q}_{h,\A,V_{h}}(\chi_{1}\psi)-\tilde Ch^{2-2\rho}\|\chi_{1}\psi\|_{\sL^2(\Omega)}^2\geq \frac{1}{2}\mathcal{Q}_{h,\A}(\chi_{1}\psi) -Ch^{2-\frac{2}{p}}\int_{\Omega} |\psi|^{p-2}\left|\chi_{1}\psi\right|^2\dx\x.$$
Moreover, since $\psi$ is $\sL^p$-normalized, we have
$$\int_{\Omega} |\psi|^{p-2}\left|\chi_{1}\psi\right|^2\dx\x\leq C\left\|\chi_{1}\psi\right\|_{\sL^p(\Omega)}^2.$$
By using again the rescaled Sobolev embedding ($\eps=\delta h^2$, with $\delta$ small enough) and the diamagnetic inequality, we infer
$$\left\|\chi_{1}\psi\right\|_{\sL^p(\Omega)}^2\leq C\left((\delta h^2)^{1-\frac{p}{2}}\|\chi_{1}\psi\|^2_{\sL^2(\Omega)}+\delta\mathcal{Q}_{h,\A}(\chi_{1}\psi)\right).$$
Therefore, it follows that
$$\eta h\|\chi_{2}e^{\chi h^{-\rho}|\x|}\psi\|_{\sL^2(\Omega)}^2\leq C(\delta h^2)^{1-\frac{p}{2}}\|\psi\|^2_{\sL^2(\Omega)}.$$
and thus
$$\|e^{\chi h^{-\rho}|\x|}\psi\|_{\sL^2(\Omega)}^2\leq \tilde Ch^{1-p}\|\psi\|^2_{\sL^2(\Omega)}.$$
With the previous analysis, we also infer that
$$\mathcal{Q}_{h,\A}(\chi_{2}e^{\chi h^{-\rho}|\x|}\psi)\leq Ch^{-\gamma}\|\psi\|^2_{\sL^2(\Omega)},$$
for some $\gamma>0$. 
With the same kind of arguments, we have
$$\mathcal{Q}_{h,\A}(\chi_{1}e^{\chi h^{-\rho}|\x|}\psi)\leq C\mathcal{Q}_{h,\A,V_{h}}(\chi_{1}e^{\chi h^{-\rho}|\x|}\psi)+Ch^{-\gamma}\|\psi\|^2_{\sL^2(\Omega)}\leq\tilde Ch^{-\gamma}\|\psi\|^2_{\sL^2(\Omega)}.$$
With the \enquote{IMS} formula, we find
$$\mathcal{Q}_{h,\A}(e^{\chi h^{-\rho}|\x|}\psi)\leq Ch^{-\gamma}\|\psi\|^2_{\sL^2(\Omega)}.$$
Since $\Omega$ is bounded, $\A$ is regular and changing $\gamma$, we get
$$\|\nabla\left(e^{\chi h^{-\rho}|\x|}\psi\right)\|_{\sL^2(\Omega)}^2\leq Ch^{-\gamma}\|\psi\|^2_{\sL^2(\Omega)}$$
and then,
$$\|e^{\chi h^{-\rho}|\x|}\psi\|^2_{\sH^1(\Omega)}\leq Ch^{-\gamma}\|\psi\|^2_{\sL^2(\Omega)},\quad \|e^{\chi h^{-\rho}|\x|}\psi\|^2_{\sL^q(\Omega)}\leq Ch^{-\gamma}\|\psi\|^2_{\sL^2(\Omega)}\leq \tilde Ch^{-\gamma}\|\psi\|^2_{\sL^q(\Omega)},$$
the second inequality coming from the Sobolev embedding for all $q\geq 2$ and the H\"older inequality ($\Omega$ is bounded). Finally, using the equation satisfied by $\psi$, we infer that 
$$\|e^{\chi h^{-\rho}|\x|}\psi\|^2_{\sH^2(\Omega)}\leq Ch^{-\gamma}\|\psi\|^2_{\sL^2(\Omega)}$$
and thus
$$\|e^{\chi h^{-\rho}|\x|}\psi\|^2_{\sL^\infty(\Omega)}\leq Ch^{-\gamma}\|\psi\|^2_{\sL^2(\Omega)}.$$
This finishes the proof of Theorem~\ref{loc}.

\subsection{Lower bound: vanishing magnetic field}
This section is devoted to the proof of the lower bound in Theorem \ref{theo2}. Let us first state a convenient lemma.
\begin{lemma}\label{min-param}
For $b, c_{1}, c_{2}\in\R$, with $c_{2}\neq 0$, let us introduce
$$A_{b,c_{1}, c_{2}}(s,t)=-b t+c_{1}st+\frac{c_{2}}{2}t^2,\qquad \A_{b, c_{1}, c_{2}}=(A_{b,c_{1}, c_{2}},0).$$
For $p\geq2$, we consider
$$\mu(b, c_{1}, c_{2},p)=\inf_{\psi\in\sH_{\A}^1, \psi\neq 0}\frac{\int_{\R^2} |D_{t}\psi|^2+|(D_{s}+A_{b,c_{1}, c_{2}}(s,t))\psi|^2\dx s\dx t}{\|\psi\|^2_{\sL^p(\R^2)}}.$$
Then, we have (with $\| c \|$ being the Euclidean norm of $(c_1,c_2)$)
$$\mu(b, c_{1}, c_{2},p)=\mu(0, c_{1}, c_{2},p)=\|c\|^{\frac{4}{3p}}\lambda^{[1]}(p ).$$
Moreover, for $p>2$, the infimum is a minimum.
\end{lemma}
\begin{proof}
It is enough to observe that
$$A_{b,c_{1}, c_{2}}(s,t)=c_{1}st+\frac{c_{2}}{2}\left(t-\frac{b}{c_{2}}\right)^2-\frac{b^2}{2c_{2}}$$
and to consider the translation $\tau=t-\frac{b}{c_{2}}$. Then, a change of gauge, a rotation and a rescaling provide the conclusion.
\end{proof}

\begin{proof}[Proof of the lower bound in Theorem \ref{theo2}]
We can now prove the lower bound announced in Theorem \ref{theo2}. Let us again consider $\psi\in\Dom(\mathcal{Q}_{h,\A})$ and use our partition of the unity \eqref{IMS-lb}. We denote by $(\x_{\ell})$ the centers of the balls and define $b_{\ell} = \B(\x_{\ell})$. We distinguish between the balls that are close to $\Gamma$ and the others by letting
\begin{align*}
&J_{1}(h)=\{\ell : \dist(\x_{\ell},\Gamma)\geq\eps_{0}\},\\
&J_{2}(h)=\{\ell : h^{\tilde\rho}<\dist(\x_{\ell},\Gamma)<\eps_{0}\},\\
&J_{3}(h)=\{\ell : \dist(\x_{\ell},\Gamma)\leq h^{\tilde\rho}\}.
\end{align*}
Here $\eps_0$ is chosen so small that the local coordinates $(s,t)$ introduced around \eqref{eq:locCOord} make sense in the regions covered by $J_2(h)$ and $J_3(h)$.

\subsubsection{Collecting the balls of the region $J_{1}(h)$}
Using that the magnetic field does not vanish in the region determined by $J_{1}(h)$ and using \eqref{HM}, we find first
\begin{multline*}
\sum_{\ell\in J_{1}(h)}\left\{\mathcal{Q}_{h,\A}(\tilde\chi_{\alpha,\rho,h}^{[\ell]}\psi) -Dh^{2-2\alpha}\|\tilde\chi_{\alpha,\rho,h}^{[\ell]}\psi\|^2_{\sL^2(\Omega)}\right\}\geq\sum_{\ell\in J_{1}(h)}(1-\tilde Dh^{1-2\alpha})\mathcal{Q}_{h,\A}(\tilde\chi_{\alpha,\rho,h}^{[\ell]}\psi),
\end{multline*}
and then, with the lower bound of Theorem \ref{theo1} (in the region where the magnetic field is bounded from below by a positive constant),
\begin{multline*}
\sum_{\ell\in J_{1}(h)}\left\{\mathcal{Q}_{h,\A}(\tilde\chi_{\alpha,\rho,h}^{[\ell]}\psi) -Dh^{2-2\alpha}\|\tilde\chi_{\alpha,\rho,h}^{[\ell]}\psi\|^2_{\sL^2(\Omega)}\right\}\\
\geq (1-\tilde Dh^{1-2\alpha})(1-Ch^{\frac{1}{8}})\lambda^{[0]}(p )b^{\frac{2}{p}}_{0}h^2h^{-\frac{2}{p}}\|\tilde\chi_{\alpha,\rho,h}^{[\ell]}\psi\|^2_{\sL^p(\Omega)},
\end{multline*}
where $b_0>0$ is the minimal value of the magnetic field strength in the region covered by balls from $J_1(h)$.

 \bigskip
 
In the regions determined by $J_{2}(h)$ and $J_{3}(h)$, we use the tubular coordinates near the zero line $\Gamma$. 
We recall the asymptotic expansion of the linear eigenvalue (see \cite{HelMo96, DomRay13}) under the assumptions of Theorem \ref{theo2}:
\begin{equation}\label{lin-asymptotics}
\lambda(\Omega, \A, 2,h)=\gamma_{0}^{\frac{2}{3}}\lambda^{[1]}(2)h^{\frac{4}{3}}+o(h^{\frac{4}{3}}).
\end{equation}
Therefore, since we have
$$\lambda(\Omega, \A, 2,h)\|\tilde\chi_{\alpha,\rho,h}^{[\ell]}\psi\|^2_{\sL^2(\Omega)}\leq \mathcal{Q}_{h,\A}(\tilde\chi_{\alpha,\rho,h}^{[\ell]}\psi)$$
we get, for $j\in\{2,3\}$,
\begin{equation}\label{min-Jj}
\sum_{\ell\in J_{j}(h)}\left\{\mathcal{Q}_{h,\A}(\tilde\chi_{\alpha,\rho,h}^{[\ell]}\psi) -Dh^{2-2\alpha}\|\tilde\chi_{\alpha,\rho,h}^{[\ell]}\psi\|^2_{\sL^2(\Omega)}\right\}\geq\sum_{\ell\in J_{j}(h)}(1-\tilde Dh^{\frac{2}{3}-2\alpha})\mathcal{Q}_{h,\A}(\tilde\chi_{\alpha,\rho,h}^{[\ell]}\psi).
\end{equation}

\subsubsection{Collecting the balls of the region $J_{2}(h)$}
By changing to the local coordinates introduced in \eqref{eq:locCOord} we find
\begin{align*}
\mathcal{Q}_{h,\A}(\tilde\chi_{\alpha,\rho,h}^{[\ell]}\psi)\geq
\frac{1}{2}\int |hD_{t}\widetilde{\psi_{\ell}}|^2+|(hD_{s}+A_{2,\ell}(s,t)+R_{3,\ell}(s,t))\widetilde{\psi_{\ell}}|^2 \dx s\dx t,
\end{align*}
where $\widetilde{\psi_{\ell}} = e^{i \varphi_{\ell}/h} (\tilde\chi_{\alpha,\rho,h}^{[\ell]}\psi)(\Phi(s,t))$, for some suitable local gauge transformation $\varphi_{\ell}$ and
with $$A_{2,\ell}(s,t)=-b_{\ell}(t-t_{\ell})+c_{1,\ell}(s-s_{\ell})(t-t_{\ell})+\frac{c_{2,\ell}}{2}(t-t_{\ell})^2$$
and where $R_{3,\ell}$ is Taylor remainder of order $3$ of $\tilde A$ at $(s_{\ell},t_{\ell})$.
Actually, at this point we could include the terms of order $2$ in the remainder, but since we will use the higher precision in the treatment of the $J_3(h)$-terms, we introduce the notation here.

By a Cauchy inequality and the support properties of $\tilde\chi_{\alpha,\rho,h}^{[\ell]}$,
\begin{align*}
\int &|hD_{t}\widetilde{\psi_{\ell}}|^2+|(hD_{s}+A_{2,\ell}(s,t)+R_{3,\ell}(s,t))\widetilde{\psi_{\ell}}|^2 \dx s\dx t \\
&\geq
(1-\eta) \int |hD_{t}\widetilde{\psi_{\ell}}|^2+|(hD_{s}-b_{\ell}t)\widetilde{\psi_{\ell}}|^2 \dx s\dx t - C \eta^{-1} h^{4\rho} \| \widetilde{\psi_{\ell}}\|^2.
\end{align*}
Notice that since $\ell \in J_2(h)$, $|b_{\ell}| \geq C h^{\tilde{\rho}}$. So we can estimate (using \eqref{eq:MagnUncertain})
\begin{align*}
\int |hD_{t}\widetilde{\psi_{\ell}}|^2+|(hD_{s}-b_{\ell}t)\widetilde{\psi_{\ell}}|^2 \dx s\dx t \geq
h b_{\ell} \| \widetilde{\psi_{\ell}}\|_{\sL^2}^2 \geq C h^{1+ \tilde{\rho}} \| \widetilde{\psi_{\ell}}\|_{\sL^2}^2.
\end{align*}
We deduce that
\begin{align*}
\mathcal{Q}_{h,\A}(\tilde\chi_{\alpha,\rho,h}^{[\ell]}\psi)\geq
\left(\frac{1}{2}(1-\eta)-C\eta^{-1}h^{4\rho-1-\tilde\rho}\right)\int |hD_{t}\widetilde{\psi_{\ell}}|^2+|(hD_{s}-b_{\ell}t)\widetilde{\psi_{\ell}}|^2 \dx s\dx t.
\end{align*}
Choosing $\eta=h^{2\rho-\frac{1}{2}-\frac{\tilde\rho}{2}}$, 
\begin{align*}
\mathcal{Q}_{h,\A}(\tilde\chi_{\alpha,\rho,h}^{[\ell]}\psi)\geq
\left(\frac{1}{2}-Ch^{2\rho-\frac{1}{2}-\frac{\tilde\rho}{2}}\right)\int |hD_{t}\widetilde{\psi_{\ell}}|^2+|(hD_{s}-b_{\ell}t)\widetilde{\psi_{\ell}}|^2 \dx s\dx t,
\end{align*}
so that, by using a scaling argument and the definition of $\lambda^{[1]}(p)$,
\begin{align*}
\mathcal{Q}_{h,\A}(\tilde\chi_{\alpha,\rho,h}^{[\ell]}\psi)&\geq\left(\frac{1}{2}-Ch^{2\rho-\frac{1}{2}-\frac{\tilde\rho}{2}}\right)b_{\ell}^{\frac{2}{p}} h^{2-\frac{2}{p}}\lambda^{[1]}(p)\|\widetilde{\psi_{\ell}}\|^2_{\sL^p}\\
&\geq \left(\frac{1}{2}-Ch^{2\rho-\frac{1}{2}-\frac{\tilde\rho}{2}}\right)h^{\frac{2\tilde\rho}{p}} h^{2-\frac{2}{p}}\lambda^{[1]}(p)\|\widetilde{\psi_{\ell}}\|^2_{\sL^p}.
\end{align*}
We deduce that, for $h$ sufficiently small, and using also \eqref{min-Jj},
\begin{align*}
\sum_{\ell\in J_{2}(h)}\left\{\mathcal{Q}_{h,\A}(\tilde\chi_{\alpha,\rho,h}^{[\ell]}\psi) -Dh^{2-2\alpha}\|\tilde\chi_{\alpha,\rho,h}^{[\ell]}\psi\|^2_{\sL^2(\Omega)}\right\}
\geq
\frac{1}{4}\lambda^{[1]}(p) h^{\frac{2\tilde\rho}{p}} h^{2-\frac{2}{p}}\sum_{\ell\in J_{2}(h)} \|\tilde\chi_{\alpha,\rho,h}^{[\ell]}\psi\|^2_{\sL^p}.
\end{align*}
\subsubsection{Collecting the balls of the region $J_{3}(h)$}
By changing to local coordinates as in the region $J_2(h)$,
\begin{multline*}
\mathcal{Q}_{h,\A}(\tilde\chi_{\alpha,\rho,h}^{[\ell]}\psi)\geq\\
(1-Ch^{\tilde\rho})\int |hD_{t}( \widetilde{\psi_{\ell}} )|^2+|(hD_{s}+A_{2,\ell}(s,t)+R_{3,\ell}(s,t))( \widetilde{\psi_{\ell}} )|^2 \dx s\dx t.
\end{multline*}
We notice that (for any $\eta>0$),
\begin{multline*}
\int |hD_{t}( \widetilde{\psi_{\ell}} )|^2+|(hD_{s}+A_{2,\ell}(s,t)+R_{3,\ell}(s,t))( \widetilde{\psi_{\ell}} )|^2 \dx s\dx t\\
\geq (1-\eta-C\eta^{-1}h^{6\rho-\frac{4}{3}})\int |hD_{t}( \widetilde{\psi_{\ell}} )|^2+|(hD_{s}+A_{2,\ell}(s,t))( \widetilde{\psi_{\ell}} )|^2 \dx s\dx t,
\end{multline*}
where we have used
$$\int |hD_{t}( \widetilde{\psi_{\ell}} )|^2+|(hD_{s}+A_{2,\ell}(s,t))( \widetilde{\psi_{\ell}} )|^2 \dx s\dx t\geq \lambda^{[1]}( 2)\|c_{\ell}\|^{\frac{2}{3}}h^{\frac{4}{3}}\| \widetilde{\psi_{\ell}} \|_{\sL^2}^2,$$
which itself comes from Lemma \ref{min-param} with $p=2$. Thus, choosing $\eta = h^{3\rho -\frac{2}{3}}$, and using again Lemma \ref{min-param}, we get
\begin{align*}
&\mathcal{Q}_{h,\A}(\tilde\chi_{\alpha,\rho,h}^{[\ell]}\psi)\\
&\geq(1-Ch^{\tilde\rho})(1-Ch^{3\rho-\frac{2}{3}})\int |hD_{t}( \widetilde{\psi_{\ell}} )|^2+|(hD_{s}+A_{2,\ell}(s,t))( \widetilde{\psi_{\ell}} )|^2 \dx s\dx t\\
&\geq (1-Ch^{\tilde\rho})(1-Ch^{3\rho-\frac{2}{3}})\|c_{\ell}\|^{\frac{4}{3p}}\lambda^{[1]}(p )\| \widetilde{\psi_{\ell}} \|^2_{\sL^p}\\
&\geq (1-\tilde Ch^{\tilde\rho})(1-Ch^{3\rho-\frac{2}{3}})\gamma_{0}^{\frac{4}{3p}}\lambda^{[1]}(p )\| \widetilde{\psi_{\ell}} \|^2_{\sL^p},
\end{align*}
where the last inequality will be justified below:
Notice that 
$$
c_{1,\ell} = \partial_{s} \tilde{\B}(s,t)|_{(s_{\ell}, t_{\ell})} = \nabla \B( \Phi(s_{\ell}, t_{\ell})) \cdot [ c'(s_{\ell}) + t_{\ell} \partial_s \n(s_{\ell})]
= {\mathcal O}(h^{\tilde{\rho}}),
$$
since $\nabla \B( \Phi(s_{\ell}, 0))  \cdot c'(s_{\ell})=0$. Similarly,
$$
c_{1,\ell} = \partial_{t} \tilde{\B}(s,t)|_{(s_{\ell}, t_{\ell})} = \nabla \B( \Phi(s_{\ell}, t_{\ell})) \cdot  \n(s_{\ell}) =  \partial_{\n,\Gamma}\B(x_{\ell})
+ {\mathcal O}(h^{\tilde{\rho}}).
$$
This gives the inequality.

It follows that
\begin{multline*}
\sum_{\ell\in J_{3}(h)}\left\{\mathcal{Q}_{h,\A}(\tilde\chi_{\alpha,\rho,h}^{[\ell]}\psi) -Dh^{2-2\alpha}\|\tilde\chi_{\alpha,\rho,h}^{[\ell]}\psi\|^2_{\sL^2(\Omega)}\right\}\\
\geq(1-\tilde Dh^{\frac{2}{3}-2\alpha})(1-\tilde Ch^{\tilde\rho})(1-Ch^{3\rho-\frac{2}{3}})\gamma_{0}^{\frac{4}{3p}}\lambda^{[1]}(p )h^{2-\frac{4}{3p}}\sum_{\ell\in J_{3}(h)}\|\tilde\chi_{\alpha,\rho,h}^{[\ell]}\psi\|^2_{\sL^p}.
\end{multline*}
\subsubsection{Optimization of the parameters}
The constraints on the different powers of $h$ are:
\begin{equation}\label{constraints}
0<\alpha<\frac{1}{3}, \quad 0<\tilde\rho<\min\left(2\rho-\frac{1}{2},\frac{1}{3},\rho\right),\quad \frac{2}{9}<\rho<\alpha.
\end{equation}
Under these constraints, the smallest term comes from the region determined by $J_{3}(h)$ and we find
\begin{multline*}
\mathcal{Q}_{h,\A}(\psi)\geq (1-\tilde Dh^{\frac{2}{3}-2\alpha})(1-\tilde Ch^{\tilde\rho})(1-Ch^{3\rho-\frac{2}{3}})\gamma_{0}^{\frac{4}{3p}}\lambda^{[1]}(p )h^{2-\frac{4}{3p}}\sum_{\ell}\|\tilde\chi_{\alpha,\rho,h}^{[\ell]}\psi\|^2_{\sL^p}.
\end{multline*}
With Lemma \ref{Lp-partition}, we deduce
\begin{multline*}
\mathcal{Q}_{h,\A}(\psi)\geq (1-\tilde Dh^{\frac{2}{3}-2\alpha})(1-\tilde Ch^{\tilde\rho})(1-Ch^{3\rho-\frac{2}{3}})\gamma_{0}^{\frac{4}{3p}}\lambda^{[1]}(p )h^{2-\frac{4}{3p}}\sum_{\ell}\|\tilde\chi_{\alpha,\rho,h}^{[\ell]}\psi\|^2_{\sL^p}\\
\geq (1-Ch^{\alpha-\rho})(1-\tilde Dh^{\frac{2}{3}-2\alpha})(1-\tilde Ch^{\tilde\rho})(1-Ch^{3\rho-\frac{2}{3}})\gamma_{0}^{\frac{4}{3p}}\lambda^{[1]}(p )h^{2-\frac{4}{3p}}\|\psi\|_{\sL^p(\Omega)}^2.
\end{multline*}
Let us consider the case when
$$\alpha-\rho=\frac{2}{3}-2\alpha=3\rho-\frac{2}{3}$$
which provides $\alpha=\frac{10}{33}$ and $\rho=\frac{8}{33}$. Unfortunately, the second constraint in \eqref{constraints} cannot be satisfied for this choice. Therefore, we take rather
$$\rho=\frac{9}{33},\quad \alpha=\frac{10}{33}$$
and we take $\tilde\rho\in\left(0,\frac{1}{22}\right)$ so that
$$\mathcal{Q}_{h,\A}(\psi)\geq(1-Ch^{\frac{1}{33}})\gamma_{0}^{\frac{4}{3p}}\lambda^{[1]}(p )h^{2-\frac{4}{3p}}\|\psi\|_{\sL^p(\Omega)}^2.$$
\end{proof}

\appendix
\section{Concentration-compactness method}\label{A}
In this section, for the convenience of the reader, we recall the strategy used in \cite{EL89}. We consider the minimization problem (with $p>2$)
$$\lambda=\inf_{\underset{ \|\psi\|_{\sL^p(\R^2)}=1}{u\in \sH^1_{\A}(\R^2),}}Q_{\A}(u),\qquad Q_{\A}(u)=\int_{\R^2}|(-i\nabla+\A)u|^2\dx\x,$$
with a non-zero magnetic field $\B$. Let us also introduce the following norm defined on $\Dom(Q_{\A})=\sH^1_{\A}$ by $\|u\|_{\sH^1_{\A}}=\left(\|u\|^2_{\sL^2}+Q_{\A}(u)\right)^{\frac{1}{2}}$. We have $\lambda>0$ and we introduce a minimizing sequence $(u_{n})_{n\geq 0}$. Let us consider the density measure $\mu_{n}=(|u_{n}|^2+|(-i\nabla+\A)u_{n}|^2)\dx\x$ whose total mass $\mu_{n}(\R^2)$ converges to $\mu>0$, with $\|u_{n}\|_{\sL^p(\Omega)}=1$. By using a slight adaptation of \cite[Lemma I.1]{Lions84}, there are three possible behaviors for the sequence $(\mu_{n})_{n\in\N}$. 

The first one is vanishing, that is
$$\forall R>0,\quad\lim_{n\to+\infty}\sup_{\x\in\R^2}\mu_{n}(D(\x,R))=0.$$
With the diamagnetic inequality, this implies that $\sup_{\x\in\R^2}\||u_{n}|\|_{\sH^1(D(\x,R))}\to 0$. By Sobolev embedding and \cite[Lemma I.1]{Lions84b}, it follows that $\|u_{n}\|_{\sL^p}\to 0$. This is a contradiction.

The second possible behavior is dichotomy, that is
\begin{align*}
\exists \beta\in(0,\mu), \forall \eps>0, \exists R_{1}>0, \exists R_{n}\to+\infty, (y_{n})_{n\in\N},\\
|\mu_{n}(D(y_{n}, R_{1}))-\beta|\leq\eps,\quad |\mu_{n}(\complement D(y_{n}, R_{n}))-(\mu-\beta)|\leq\eps,
\end{align*}
so that
$$|\mu_{n}(D(y_{n}, R_{1})\cup \complement D(y_{n}, R_{n}))-\mu|\leq2\eps.$$
By using cutoff functions and the \enquote{IMS} formula, we can find $\chi_{n,1}$ and $\chi_{n,2}$, with supports such that $\dist(\supp(\chi_{n,1}),\supp(\chi_{n,2}))\to+\infty$ such that
\begin{equation}\label{split-mass}
\left|\|\chi_{n,1}u_{n}\|^2_{\sH^1_{\A}}-\beta\right|\leq \eps, \quad \left|\|\chi_{n,2}u_{n}\|^2_{\sH^1_{\A}}-(\mu-\beta)\right|\leq \eps,
\end{equation}
$$\|u_{n}-\chi_{n,1}u_{n}-\chi_{n,2}u_{n}\|_{\sH^1_{\A}}\leq C\eps.$$
From the last inequality, we get
\begin{equation}\label{split-Q}
\left|Q_{\A}(u_{n})-Q_{\A}(\chi_{n,1}u_{n})-Q_{\A}(\chi_{n,2}u_{n})\right|\leq C\eps
\end{equation}
and also, by Sobolev embedding,
$$\|u_{n}-\chi_{n,1}u_{n}-\chi_{n,2}u_{n}\|_{\sL^p}\leq C\eps.$$
This implies
$$|\|u_{n}\|_{\sL^p}-\|\chi_{n,1}u_{n}+\chi_{n,2}u_{n}\|_{\sL^p}|\leq C\eps$$
so that
\begin{equation}\label{split-Lp}
|\|u_{n}\|^p_{\sL^p}-\|\chi_{n,1}u_{n}\|_{\sL^p}^p-\|\chi_{n,2}u_{n}\|^p_{\sL^p}|\leq \tilde C\eps.
\end{equation}
We notice that, for $j\in{1,2}$, due to \eqref{split-mass}, 
\begin{equation}\label{positive-mass}
\liminf_{n\to+\infty} Q_{\A}(\chi_{n,j}u_{n})>0.
\end{equation}
We let $\alpha_{n}=\|\chi_{n,1}u_{n}\|_{\sL^p}^p$. We may assume that $(\alpha_{n})$ converges to some $\alpha\in[0,1]$. If $\alpha=1$, we get, by definition of $\lambda$, that
$$Q_{\A}(\chi_{n,1}u_{n})\geq \lambda\alpha_{n}^{\frac{2}{p}}\geq \lambda(1-C\eps).$$
since $(u_{n})$ is a minimizing sequence and due to \eqref{split-Q}, we have $Q_{\A}(\chi_{n,2}u_{n})\leq C\eps$ which contradicts \eqref{positive-mass}. In the same way, we have $\alpha\neq 0$.

With \eqref{split-Q} and \eqref{split-Lp}, we get
$$\lambda\geq \lambda\alpha_{n}^{\frac{2}{p}}+ \lambda(1-\alpha_{n})^{\frac{2}{p}}-C\eps$$
and thus, for $\alpha\in(0,1)$,
$$\alpha^{\frac{2}{p}}+(1-\alpha)^{\frac{2}{p}}\leq 1.$$
This is not possible when $p>2$.

The third and last possible behavior is tightness up to translation (compactness case). There exists $(\x_{n})_{n\geq 0}$ such that
\begin{equation*}
\forall \eps>0,\quad\exists R>0, \quad \forall n\geq 1,\quad \mu_{n}(\R^2\setminus D(\x_{n},R))\leq \eps.
\end{equation*}

\bibliographystyle{mnachrn}
\bibliography{bib-nonlinear}
\end{document}